\documentclass[10pt]{amsart}
\usepackage{amsmath}
\usepackage{amsxtra}
\usepackage{amscd}
\usepackage{amsthm}
\usepackage{amsfonts}
\usepackage{amssymb}
\usepackage{eucal}
\usepackage[all]{xy}
\usepackage{graphicx}
\usepackage[usenames]{color}

\newtheorem{cor}[subsubsection]{Corollary}
\newtheorem{lem}[subsubsection]{Lemma}
\newtheorem{prop}[subsubsection]{Proposition}

\newtheorem{conj}[subsubsection]{Conjecture}
\newtheorem{thm}[subsubsection]{Theorem}

\theoremstyle{remark}
\newtheorem{rem}[subsubsection]{Remark}

\theoremstyle{definition}

\theoremstyle{remark}

\newcommand{\thmref}[1]{Theorem~\ref{#1}}

\newcommand{\secref}[1]{Sect.~\ref{#1}}
\newcommand{\lemref}[1]{Lemma~\ref{#1}}
\newcommand{\propref}[1]{Proposition~\ref{#1}}
\newcommand{\corref}[1]{Corollary~\ref{#1}}
\newcommand{\conjref}[1]{Conjecture~\ref{#1}}

\numberwithin{equation}{section}

\newcommand{\nc}{\newcommand}
\nc{\renc}{\renewcommand}
\nc{\ssec}{\subsection}
\nc{\sssec}{\subsubsection}
\nc{\on}{\operatorname}

\nc{\ips}{{\iota_P^{(S)}}}
\nc{\ipms}{{\iota_{P^-}^{(S)}}}
\nc{\sfpps}{{\sfp_P^{(S)}}}
\nc{\sfppms}{{\sfp_{P^-}^{(S)}}}

\nc\ol{\overline}
\nc\wt{\widetilde}
\nc\tboxtimes{\wt{\boxtimes}}
\nc\tstar{\wt{\star}}
\nc{\alp}{\alpha}

\nc{\ZZ}{{\mathbb Z}}
\nc{\NN}{{\mathbb N}}
\nc{\BF}{{\mathbb F}}
\nc{\OO}{{\mathbb O}}
\renc{\SS}{{\mathbb S}}
\nc{\DD}{{\mathbb D}}
\nc{\GG}{{\mathbb G}}

\nc{\Fq}{{\mathbb F}_q}
\nc{\Fqb}{\ol{\mathbb F}_q}
\nc{\Ql}{{\mathbb Q}_\ell}
\nc{\Qlb}{{\ol{\mathbb Q}_\ell}}
\nc{\id}{\text{id}}
\nc\X{\mathcal X}

\nc{\red}{\on{red}}
\nc{\Ho}{\on{Ho}}
\nc{\Hom}{\on{Hom}}
\nc{\coef}{\on{coef}}
\nc{\Lie}{\on{Lie}}
\nc{\Diff}{\on{Diff}}
\nc{\Loc}{\on{Loc}}
\nc{\coLoc}{\on{coLoc}}
\nc{\Pic}{\on{Pic}}
\nc{\Bun}{\on{Bun}}
\nc{\IC}{\on{IC}}
\nc{\Aut}{\on{Aut}}
\nc{\rk}{\on{rk}}
\nc{\Sh}{\on{Sh}}
\nc{\Perv}{\on{Perv}}
\nc{\pos}{{\on{pos}}}
\nc{\Conv}{\on{Conv}}
\nc{\Sph}{\on{Sph}}
\nc{\Sym}{\on{Sym}}
\nc{\BunBb}{\overline{\Bun}_B}
\nc{\BunNb}{\overline{\Bun}_N}
\nc{\BunTb}{\overline{\Bun}_T}
\nc{\BunBbm}{\overline{\Bun}_{B^-}}
\nc{\BunBbel}{\overline{\Bun}_{B,el}}
\nc{\BunBbmel}{\overline{\Bun}_{B^-,el}}
\nc{\Buno}{\overset{o}{\Bun}}
\nc{\BunPb}{{\overline{\Bun}_P}}
\nc{\BunBM}{\Bun_{B(M)}}
\nc{\BunBMb}{\overline{\Bun}_{B(M)}}
\nc{\BunPbw}{{\widetilde{\Bun}_P}}
\nc{\BunBP}{\widetilde{\Bun}_{B,P}}
\nc{\GUb}{\overline{G/U}}
\nc{\GUPb}{\overline{G/U(P)}}

\nc{\Hhom}{\underline{\on{Hom}}}
\nc\syminfty{\on{Sym}^{\infty}}
\nc\lal{\ol{\lambda}}
\nc\xl{\ol{x}}
\nc\thl{\ol{\theta}}
\nc\nul{\ol{\nu}}
\nc\mul{\ol{\mu}}
\nc\Sum\Sigma
\nc{\oX}{\overset{\circ}{X}{}}
\nc{\wtoX}{\wt{\overset{\circ}{X}}{}}
\nc{\hl}{\overset{\leftarrow}h{}}
\nc{\hr}{\overset{\rightarrow}h{}}
\nc{\M}{{\mathcal M}}
\nc{\N}{{\mathcal N}}
\nc{\F}{{\mathcal F}}
\nc{\D}{{\mathcal D}}
\nc{\Q}{{\mathcal Q}}
\nc{\Y}{{\mathcal Y}}
\nc{\G}{{\mathcal G}}
\nc{\E}{{\mathcal E}}
\nc{\CalC}{{\mathcal C}}
\nc\Dh{\widehat{\D}}

\nc{\C}{{\mathcal C}}
\nc{\K}{{\mathcal K}}
\renewcommand{\H}{{\mathcal H}}

\nc{\T}{{\mathcal T}}
\nc{\V}{{\mathcal V}}
\renc{\P}{{\mathcal P}}
\nc{\A}{{\mathcal A}}
\nc{\B}{{\mathcal B}}
\nc{\U}{{\mathcal U}}

\nc{\Gr}{{\on{Gr}}}

\nc{\frn}{{\check{\mathfrak u}(P)}}

\nc{\fC}{\mathfrak C}
\nc{\p}{\mathfrak p}
\nc{\q}{\mathfrak q}
\nc\f{{\mathfrak f}}

\nc{\qo}{{\mathfrak q}}
\nc{\po}{{\mathfrak p}}
\nc{\s}{{\mathfrak s}}
\nc\w{\text{w}}

\renewcommand{\mod}{{\on{-mod}}}

\nc\mathi\iota
\nc\Spec{\on{Spec}}
\nc\Proj{\on{Proj}}
\nc\Mod{\on{Mod}}
\nc{\tw}{\widetilde{\mathfrak t}}
\nc{\pw}{\widetilde{\mathfrak p}}
\nc{\qw}{\widetilde{\mathfrak q}}
\nc{\jw}{\widetilde j}

\nc{\grb}{\overline{\Gr}}
\nc{\I}{\mathcal I}

\nc{\lambdach}{{\check\lambda}}
\nc{\Lambdach}{{\check\Lambda}{}}
\nc{\much}{{\check\mu}}
\nc{\omegach}{{\check\omega}}
\nc{\nuch}{{\check\nu}}
\nc{\etach}{{\check\eta}}
\nc{\alphach}{{\check\alpha}}
\nc{\oblvtach}{{\check\oblvta}}
\nc{\rhoch}{{\check\rho}}
\nc{\ch}{{\check h}}

\nc{\Hb}{\overline{\H}}

\nc{\BA}{{\mathbb{A}}}
\nc{\BC}{{\mathbb{C}}}
\nc{\BG}{{\mathbb{G}}}
\nc{\BM}{{\mathbb{M}}}
\nc{\BO}{{\mathbb{O}}}
\nc{\BD}{{\mathbb{D}}}
\nc{\BN}{{\mathbb{N}}}
\nc{\BP}{{\mathbb{P}}}
\nc{\BQ}{{\mathbb{Q}}}
\nc{\BR}{{\mathbb{R}}}
\nc{\BZ}{{\mathbb{Z}}}
\nc{\BS}{{\mathbb{S}}}
\nc{\Deep}{{\bf{deep}}}
\nc{\deep}{deep}

\nc{\CA}{{\mathcal{A}}}
\nc{\CB}{{\mathcal{B}}}

\nc{\CE}{{\mathcal{E}}}
\nc{\CF}{{\mathcal{F}}}
\nc{\CH}{{\mathcal{H}}}

\nc{\CL}{{\mathcal{L}}}
\nc{\CC}{{\mathcal{C}}}
\nc{\CG}{{\mathcal{G}}}
\nc{\CalD}{{\mathcal{D}}}
\nc{\CM}{{\mathcal{M}}}
\nc{\CN}{{\mathcal{N}}}
\nc{\CK}{{\mathcal{K}}}
\nc{\CO}{{\mathcal{O}}}
\nc{\CP}{{\mathcal{P}}}
\nc{\CQ}{{\mathcal{Q}}}
\nc{\CR}{{\mathcal{R}}}
\nc{\CS}{{\mathcal{S}}}
\nc{\CT}{{\mathcal{T}}}
\nc{\CU}{{\mathcal{U}}}
\nc{\CV}{{\mathcal{V}}}
\nc{\CW}{{\mathcal{W}}}
\nc{\CX}{{\mathcal{X}}}
\nc{\CY}{{\mathcal{Y}}}
\nc{\CZ}{{\mathcal{Z}}}
\nc{\CI}{{\mathcal{I}}}

\nc{\csM}{{\check{\mathcal A}}{}}
\nc{\oM}{{\overset{\circ}{\mathcal M}}{}}
\nc{\obM}{{\overset{\circ}{\mathbf M}}{}}
\nc{\oCA}{{\overset{\circ}{\mathcal A}}{}}
\nc{\obA}{{\overset{\circ}{\mathbf A}}{}}
\nc{\ooM}{{\overset{\circ}{M}}{}}
\nc{\osM}{{\overset{\circ}{\mathsf M}}{}}
\nc{\vM}{{\overset{\bullet}{\mathcal M}}{}}
\nc{\nM}{{\underset{\bullet}{\mathcal M}}{}}
\nc{\oD}{{\overset{\circ}{\mathcal D}}{}}
\nc{\obD}{{\overset{\circ}{\mathbf D}}{}}
\nc{\oA}{{\overset{\circ}{\mathbb A}}{}}
\nc{\op}{{\overset{\bullet}{\mathbf p}}{}}
\nc{\cp}{{\overset{\circ}{\mathbf p}}{}}
\nc{\oU}{{\overset{\bullet}{\mathcal U}}{}}
\nc{\oZ}{{\overset{\circ}{\mathcal Z}}{}}
\nc{\ofZ}{{\overset{\circ}{\mathfrak Z}}{}}
\nc{\oF}{{\overset{\circ}{\fF}}}

\nc{\fa}{{\mathfrak{a}}}
\nc{\fb}{{\mathfrak{b}}}
\nc{\fd}{{\mathfrak{d}}}
\nc{\ff}{{\mathfrak{f}}}
\nc{\fg}{{\mathfrak{g}}}
\nc{\fgl}{{\mathfrak{gl}}}
\nc{\fh}{{\mathfrak{h}}}
\nc{\fj}{{\mathfrak{j}}}
\nc{\fk}{{\mathfrak{k}}}
\nc{\fl}{{\mathfrak{l}}}
\nc{\fm}{{\mathfrak{m}}}
\nc{\fn}{{\mathfrak{n}}}
\nc{\fu}{{\mathfrak{u}}}
\nc{\fp}{{\mathfrak{p}}}
\nc{\fr}{{\mathfrak{r}}}
\nc{\fs}{{\mathfrak{s}}}
\nc{\ft}{{\mathfrak{t}}}
\nc{\fv}{{\mathfrak{v}}}
\nc{\fz}{{\mathfrak{z}}}
\nc{\fsl}{{\mathfrak{sl}}}
\nc{\hsl}{{\widehat{\mathfrak{sl}}}}
\nc{\hgl}{{\widehat{\mathfrak{gl}}}}
\nc{\hg}{{\widehat{\mathfrak{g}}}}
\nc{\chg}{{\widehat{\mathfrak{g}}}{}^\vee}
\nc{\hn}{{\widehat{\mathfrak{n}}}}
\nc{\chn}{{\widehat{\mathfrak{n}}}{}^\vee}

\nc{\fA}{{\mathfrak{A}}}
\nc{\fB}{{\mathfrak{B}}}
\nc{\fD}{{\mathfrak{D}}}
\nc{\fE}{{\mathfrak{E}}}
\nc{\fF}{{\mathfrak{F}}}
\nc{\fG}{{\mathfrak{G}}}
\nc{\fK}{{\mathfrak{K}}}
\nc{\fL}{{\mathfrak{L}}}
\nc{\fM}{{\mathfrak{M}}}
\nc{\fN}{{\mathfrak{N}}}
\nc{\fP}{{\mathfrak{P}}}
\nc{\fU}{{\mathfrak{U}}}
\nc{\fV}{{\mathfrak{V}}}
\nc{\fZ}{{\mathfrak{Z}}}

\nc{\bb}{{\mathbf{b}}}
\nc{\bc}{{\mathbf{c}}}
\nc{\bd}{{\mathbf{d}}}
\nc{\bbf}{{\mathbf{f}}}
\nc{\be}{{\mathbf{e}}}
\nc{\bi}{{\mathbf{i}}}
\nc{\bj}{{\mathbf{j}}}
\nc{\bn}{{\mathbf{n}}}
\nc{\bo}{{\mathbf{o}}}
\nc{\bp}{{\mathbf{p}}}
\nc{\bq}{{\mathbf{q}}}
\nc{\bu}{{\mathbf{u}}}
\nc{\bv}{{\mathbf{v}}}
\nc{\bx}{{\mathbf{x}}}
\nc{\bs}{{\mathbf{s}}}
\nc{\by}{{\mathbf{y}}}
\nc{\bw}{{\mathbf{w}}}
\nc{\bA}{{\mathbf{A}}}
\nc{\bK}{{\mathbf{K}}}
\nc{\bB}{{\mathbf{B}}}
\nc{\bC}{{\mathbf{C}}}
\nc{\bG}{{\mathbf{G}}}
\nc{\bD}{{\mathbf{D}}}
\nc{\bE}{{\mathbf{E}}}
\nc{\bH}{{\mathbf{H}}}
\nc{\bM}{{\mathbf{M}}}
\nc{\bN}{{\mathbf{N}}}
\nc{\bO}{{\mathbf{O}}}
\nc{\bP}{{\mathbf{P}}}
\nc{\bV}{{\mathbf{V}}}
\nc{\bW}{{\mathbf{W}}}
\nc{\bX}{{\mathbf{X}}}
\nc{\bZ}{{\mathbf{Z}}}
\nc{\bS}{{\mathbf{S}}}

\nc{\sA}{{\mathsf{A}}}
\nc{\sB}{{\mathsf{B}}}
\nc{\sC}{{\mathsf{C}}}
\nc{\sD}{{\mathsf{D}}}
\nc{\sF}{{\mathsf{F}}}
\nc{\sG}{{\mathsf{G}}}
\nc{\sH}{{\mathsf{H}}}
\nc{\sK}{{\mathsf{K}}}
\nc{\sk}{{\mathsf{k}}}
\nc{\sM}{{\mathsf{M}}}
\nc{\sO}{{\mathsf{O}}}
\nc{\sW}{{\mathsf{W}}}
\nc{\sQ}{{\mathsf{Q}}}
\nc{\sP}{{\mathsf{P}}}
\nc{\sR}{{\mathsf{R}}}
\nc{\sZ}{{\mathsf{Z}}}
\nc{\sfp}{{\mathsf{p}}}
\nc{\sfq}{{\mathsf{q}}}
\nc{\sr}{{\mathsf{r}}}
\nc{\bk}{{\mathsf{k}}}
\nc{\sg}{{\mathsf{g}}}
\nc{\sff}{{\mathsf{f}}}
\nc{\sfb}{{\mathsf{b}}}
\nc{\sfc}{{\mathsf{c}}}
\nc{\sd}{{\mathsf{d}}}

\nc{\BK}{{\bar{K}}}

\nc{\tA}{{\widetilde{\mathbf{A}}}}
\nc{\tB}{{\widetilde{\mathcal{B}}}}
\nc{\tg}{{\widetilde{\mathfrak{g}}}}
\nc{\tG}{{\widetilde{G}}}
\nc{\TM}{{\widetilde{\mathbb{M}}}{}}
\nc{\tO}{{\widetilde{\mathsf{O}}}{}}
\nc{\tU}{{\widetilde{\mathfrak{U}}}{}}
\nc{\TZ}{{\tilde{Z}}}
\nc{\tx}{{\tilde{x}}}
\nc{\tbv}{{\tilde{\bv}}}
\nc{\tfP}{{\widetilde{\mathfrak{P}}}{}}
\nc{\tz}{{\tilde{\zeta}}}
\nc{\tmu}{{\tilde{\mu}}}

\nc{\urho}{\underline{\rho}}
\nc{\uB}{\underline{B}}
\nc{\uC}{{\underline{\mathbb{C}}}}
\nc{\ui}{\underline{i}}
\nc{\uj}{\underline{j}}
\nc{\ofP}{{\overline{\mathfrak{P}}}}
\nc{\oB}{{\overline{\mathcal{B}}}}
\nc{\og}{{\overline{\mathfrak{g}}}}
\nc{\oI}{{\overline{I}}}

\nc{\eps}{\varepsilon}
\nc{\hrho}{{\hat{\rho}}}

\nc{\one}{{\mathbf{1}}}
\nc{\two}{{\mathbf{t}}}

\nc{\Rep}{{\mathop{\operatorname{\rm Rep}}}}
\nc{\Tot}{{\mathop{\operatorname{\rm Tot}}}}
\nc{\Ker}{{\mathop{\operatorname{\rm Ker}}}}
\nc{\im}{{\mathop{\operatorname{\rm Im}}}}
\nc{\Hilb}{{\mathop{\operatorname{\rm Hilb}}}}
\nc{\End}{{\mathop{\operatorname{\rm End}}}}
\nc{\Ext}{{\mathop{\operatorname{\rm Ext}}}}
\nc{\CHom}{{\mathop{\operatorname{{\mathcal{H}}\it om}}}}
\nc{\GL}{{\mathop{\operatorname{\rm GL}}}}
\nc{\gr}{{\mathop{\operatorname{\rm gr}}}}
\nc{\HN}{{\mathop{\operatorname{\rm HN}}}}
\nc{\Id}{{\mathop{\operatorname{\rm Id}}}}
\nc{\de}{{\mathop{\operatorname{\rm def}}}}
\nc{\length}{{\mathop{\operatorname{\rm length}}}}
\nc{\supp}{{\mathop{\operatorname{\rm supp}}}}

\nc{\Cliff}{{\mathsf{Cliff}}}
\nc{\Fl}{\on{Fl}}
\nc{\Fib}{{\mathsf{Fib}}}
\nc{\Coh}{{\on{Coh}}}
\nc{\QCoh}{{\on{QCoh}}}
\nc{\IndCoh}{{\on{IndCoh}}}
\nc{\FCoh}{{\mathsf{FCoh}}}

\nc{\reg}{{\text{\rm reg}}}

\nc{\cplus}{{\mathbf{C}_+}}
\nc{\cminus}{{\mathbf{C}_-}}
\nc{\cthree}{{\mathbf{C}_\bullet}}
\nc{\Qbar}{{\bar{Q}}}
\nc\Eis{\on{Eis}}
\nc\Eisb{\ol\Eis{}}
\nc\Eisr{\on{Eis}^{rat}{}}
\nc\wh{\widehat}
\nc{\Def}{\on{Def_{\check{\fb}}(E)}}
\nc{\barZ}{\overline{Z}{}}
\nc{\barbarZ}{\overline{\barZ}{}}
\nc{\barpi}{\overline\pi}
\nc{\barbarpi}{\overline\barpi}
\nc{\barpip}{\overline\pi{}^+}
\nc{\barpim}{\overline\pi{}^-}

\nc{\fq}{\mathfrak q}

\nc{\fqb}{\ol{\sfq}{}}
\nc{\fpb}{\ol{\sfp}{}}
\nc{\fpr}{{\sfp^{rat}}{}}
\nc{\fqr}{{\sfq^{rat}}{}}

\nc{\hattimes}{\wh\otimes}

\nc{\bh}{{\bar{h}}}
\nc{\bOmega}{{\overline{\Omega(\check \fn)}}}

\nc{\seq}[1]{\stackrel{#1}{\sim}}

\nc{\cT}{{\check{T}}}
\nc{\cG}{{\check{G}}}
\nc{\cM}{{\check{M}}}
\nc{\cB}{{\check{B}}}

\nc{\ct}{{\check{\mathfrak t}}}
\nc{\cg}{{\check{\fg}}}
\nc{\cb}{{\check{\fb}}}
\nc{\cn}{{\check{\fn}}}

\nc{\cLambda}{{\check\Lambda}}

\nc{\cla}{{\check\lambda}}
\nc{\cmu}{{\check\mu}}
\nc{\cnu}{{\check\nu}}
\nc{\ceta}{{\check\eta}}

\nc{\DefbE}{{\on{Def}_{\cB}(E_\cT)}}

\nc{\imathb}{{\ol{\imath}}}
\nc{\rlr}{\overset{\longrightarrow}{\underset{\longrightarrow}\longleftarrow}}

\nc{\oBun}{\overset{\circ}\Bun}
\nc{\LocSys}{\on{LocSys}}
\nc{\BunBbb}{\ol{\ol{Bun}}_B}
\nc{\BunBr}{\Bun_B^{rat}}
\nc{\BunBrsg}{\Bun_B^{rat,\on{s.g.}}}
\nc{\BunBrp}{\Bun_B^{rat,polar}}
\nc{\BunBrpbg}{\Bun_B^{rat,polar,\on{b.g.}}}
\nc{\BunBrpsg}{\Bun_B^{rat,polar,\on{s.g.}}}
\nc{\BunTrp}{\Bun_T^{rat,polar}}
\nc{\BunTrpbg}{\Bun_T^{rat,polar,\on{b.g.}}}
\nc{\BunTrpsg}{\Bun_T^{rat,polar,\on{s.g.}}}
\nc{\BunNr}{\Bun_N^{rat}}
\nc{\BunNre}{\Bun_N^{enh,rat}}
\nc{\BunTr}{\Bun_T^{rat}}
\nc{\Vect}{{\on{Vect}}}
\nc{\Whit}{\on{Whit}}
\nc{\CTb}{\ol{\on{CT}}}
\nc{\Ran}{\on{Ran}}
\nc{\CTr}{\on{CT}^{rat}{}}
\nc\jmathr{\jmath^{rat}{}}
\nc{\ux}{\underline{x}}
\nc{\clambda}{{\check\lambda}}
\nc{\calpha}{{\check\alpha}}
\nc{\ind}{{\mathbf{ind}}}
\nc{\oblv}{{\mathbf{oblv}}}
\nc{\ox}{{\overline{x}}}
\nc{\cLa}{\check{\Lambda}}
\nc{\StinftyCat}{\on{DGCat}}
\nc{\inftyCat}{\infty\on{-Cat}}
\nc{\inftygroup}{\infty\on{-Grpd}}
\nc{\Dmod}{\on{D-mod}}
\nc{\CMaps}{{\mathcal Maps}}
\nc{\Maps}{\on{Maps}}
\nc{\affSch}{\on{Sch}^{\on{aff}}}
\nc{\dr}{{\on{dR}}}
\nc{\oCY}{\overset{\circ}\CY}
\nc{\leqG}{\underset{G}\leq}
\nc{\leqM}{\underset{M}\leq}
\nc{\leqGad}{\underset{G_{ad}}\leq}
\nc{\leqMad}{\underset{M_{ad}}\leq}
\nc{\Tr}{\on{Tr}}
\nc{\Frob}{\on{Frob}}
\nc{\DGCat}{\on{DGCat}}
\nc{\tDGCat}{2\on{-DGCat}}
\nc{\ev}{\on{ev}}
\nc{\mmod}{\on{-}\mathbf{mod}}
\nc{\sotimes}{\overset{!}\otimes}
\nc{\Sht}{{\on{Sht}}}
\nc{\Res}{{\on{Res}}}
\nc{\Av}{{\on{Av}}}
\nc{\Ind}{{\on{Ind}}}
\nc{\coInd}{{\on{coInd}}}
\nc{\ul}{\underline}
\nc{\Se}{\on{Se}}
\nc{\Ps}{\on{Ps}}
\nc{\Shv}{\on{Shv}}
\nc{\PsId}{\on{Ps-Id}}
\nc{\Funct}{\on{Funct}}

\title[An analog of the Deligne-Lusztig duality for $(\fg,K)$-modules]{An analog of the Deligne-Lusztig duality for $(\fg,K)$-modules}

\author{Dennis Gaitsgory and Alexander Yom Din}

\begin{document}

\maketitle

\tableofcontents

\section*{Introduction}

\ssec{Pseudo-identity, Deligne-Lustig functor and dualities}

\sssec{}  Recently, a number of papers have appeared where connections were found between the following objects:

\medskip

\noindent--The Deligne-Lusztig functor on the category of representations of a $\fp$-adic group;

\medskip

\noindent--The composition of contragredient and cohomological dualities;

\medskip

\noindent--The pseudo-identity functor on the category of D-modules/sheaves on $\Bun_G$.

\medskip

\noindent--The ``strange" operator of Drinfeld-Wang that acts on the space of automorphic functions.  

\medskip

Let us explain what these relations are. 

\sssec{}  \label{sss:intro DL}

First, in the paper \cite{BBK}, the authors consider the (derived) category $G(\bK)\mod$ of (say, admissible) representations of a $\fp$-adic group $G(\bK)$
(here $G$ is a reductive group and $\bK$ is a non-archimedian local field).
The Deligne-Lusztig functor is defined by sending a representation $\CM$ to the complex
\begin{equation} \label{e:DL}
\on{DL}(\CM):=\CM \to \underset{P}\oplus\, i^G_P\circ r^G_P(\CM) \to...\to i^G_B\circ r^G_B(\CM),
\end{equation} 
where $(r^G_P,i^G_P)$ is the adjoint pair corresponding to parabolic induction and Jacquet functor
(for a parabolic $P$), and where in the $k$-th term of the complex, the direct sum is taken over parabolics
of co-rank $k$. 

\medskip

The main theorem in that paper says that the functor $\on{DL}$ is canonically isomorphic to the composition of
\emph{contragredient} and \emph{cohomological dualities}, i.e., 
$$\on{DL}\simeq \BD^{\on{coh}}\circ \BD^{\on{contr}},$$
where $\BD^{\on{cont}}$ sends $\CM$ to its admissible dual $\CM^\vee$, and
\begin{equation} \label{e:coh dual}
\BD^{\on{coh}}(\CM):= \on{RHom}_\bG(\CM,\CH),
\end{equation} 
where $\CH$ is the regular representation of the Hecke algebra (i.e., the space of compactly supported smooth functions on $G(\bK)$).  

\medskip

We also note the following:

\medskip

\noindent(1) It is more or less tautological that the composition 
$$ \BD^{\on{contr}}\circ \BD^{\on{coh}}:G(\bK)\mod\to G(\bK)\mod$$
is isomorphic to the \emph{Serre functor} $\Se_{G(\bK)\mod}$ on $G(\bK)\mod$ (see \secref{ss:Se} for what we mean by \emph{the} Serre functor).

\medskip

Thus, one can reformulate the main result of \cite{BBK} as saying that the functors $\on{DL}$ and $\Se_{\bG\mod}$
are mutually inverse. 

\medskip

\noindent(2) The key idea in the proof of this theorem is to use the De Concino-Procesi (a.k.a., \emph{wonderful}) compactification $\ol{G}$ of $G$. 

\medskip

\noindent(3) The functors $\on{DL}$ for $G$ and the Levi $M$ corresponding to a given parabolic 
make the following diagram commute (up to a cohomological shift):
\begin{equation} \label{e:DL and ind}
\CD
G(\bK)\mod  @>{\on{DL}_G}>> G(\bK)\mod  \\
@A{i^G_P}AA  @AA{i^G_{P^-}}A  \\
M(\bK)\mod @>{\on{DL}_M}>> M(\bK)\mod,
\endCD
\end{equation} 
where $i^G_{P^-}$ is the induction functor, taken with respect to the opposite parabolic. 

\sssec{}  \label{sss:intro mirac}

Second, the paper \cite{Ga2} studies the category of D-modules/sheaves on the moduli stack $\Bun_G$ of principal $G$-bundles over
a global curve. 

\medskip

Since $\Bun_G$ is \emph{not} quasi-compact the phenomenon of ``divergence at infinity" must be taken into account. One
considers two versions of the (derived) category of D-modules/sheaves:
$$\Shv_0(\Bun_G)  \text { and } \Shv(\Bun_G),$$
where the latter is the (naturally defined) category of all D-modules/sheaves, and the former is the full subcategory that consists
of objects that are !-extended from quasi-compact open substacks.

\medskip

An arbitrary object in $\Shv(\Bun_G\times \Bun_G)$ defines a functor
$$\Shv_0(\Bun_G)  \to \Shv(\Bun_G),$$
and let us temporarily fix the conventions\footnote{This choice is made so that it is easy to make a connection 
with the papers \cite{DW,Wa}. However, in the main body of the paper, our conventions will be Verdier dual
to the ones above.} so that the object
$$(\Delta_{\Bun_G})_!(k_{\Bun_G})\in \Shv(\Bun_G\times \Bun_G)$$
(here $k_{\Bun_G}$ denotes the constant sheaf on $\Bun_G$) defines the tautological embedding 
$$\Shv_0(\Bun_G)  \hookrightarrow \Shv(\Bun_G)$$

One introduces the \emph{pseudo-identity} functor
$$\on{Ps-Id}_{\Bun_G}:\Shv_0(\Bun_G)  \to \Shv(\Bun_G)$$
to be the one given by the object
\begin{equation} \label{e:* sheaf}
(\Delta_{\Bun_G})_*(k_{\Bun_G})\in \Shv(\Bun_G\times \Bun_G).
\end{equation}

The point here is that since $\Bun_G$ is a stack and not a scheme, the diagonal map is not a closed embedding,
so that functors $(\Delta_{\Bun_G})_*$ and $(\Delta_{\Bun_G})_!$ are different. This definition makes sense not just
for $\Bun_G$, but for an arbitrary algebraic stack $\CY$. 

\medskip

The main result of the paper \cite{Ga2} is that the functor $\on{Ps-Id}_{\Bun_G}$, combined with usual Verdier
duality can be extended to an equivalence
$$\Shv(\Bun_G)^\vee\simeq \Shv(\Bun_G),$$
where $\bC\mapsto \bC^\vee$ is the operation of passage to the dual category\footnote{The above assertion should
be taken literally when $\Shv(-)$ is understood as $\Dmod(-)$, while appropriate modifications need to be
made in other sheaf-theoretic contexts.} (see \secref{sss:DG categ} below). 

\medskip

A few remarks are in order:

\medskip

\noindent(i) The object $(\Delta_{\Bun_G})_*(k_{\Bun_G})$ defining $\on{Ps-Id}_{\Bun_G}$ can be described
using the wonderful compactification $\ol{G}$ of $G$;

\medskip

\noindent(ii) Using (i), one can express $\on{Ps-Id}_{\Bun_G}$ as a complex whose
terms are compositions of Constant Term, Eisenstein functors and certain intertwining functors
$$\Shv_0(\Bun_G)\overset{\on{CT}^G_P}\longrightarrow \Shv(\Bun_M)
\overset{\Upsilon}\longrightarrow \Shv(\Bun_M) \overset{\Eis^G_{P^-}}\longrightarrow \Shv(\Bun_G).$$

\medskip

\noindent(iii) One of the key ingredients in the proof of the main result of \cite{Ga2} is the
commutativity of the following diagram (up to a cohomological shift):

\begin{equation} \label{e:funct equation}
\CD
\Shv_0(\Bun_G)  @>{\on{Ps-Id}_{\Bun_G}}>>  \Shv(\Bun_G) \\
@A{\wt\Eis^G_P}AA   @AA{\Eis^G_{P^-}}A \\
\Shv_0(\Bun_M)  @>{\on{Ps-Id}_{\Bun_M}}>>  \Shv(\Bun_M),
\endCD
\end{equation}
where $\wt\Eis^G_P$ is the Verdier-conjugate of the Eisenstein operator $\Eis^G_P$.

\sssec{}

Third, the papers \cite{DW} and \cite{Wa} consider a global function field $K$, and the spaces 
$$\on{Funct}^{\on{sm}}_0(G(\BA)/G(K)) \text{ and } \on{Funct}^{\on{sm}}(G(\BA)/G(K))$$
of compactly supported (resp., all) smooth functions on the automorphic quotient $G(\BA)/G(K)$.  

\medskip

Using the wonderful compactification $\ol{G}$, the authors define a certain ``strange" operator
$$L:\on{Funct}^{\on{sm}}_0(G(\BA)/G(K)) \to \on{Funct}^{\on{sm}}(G(\BA)/G(K))$$
that commutes with the $G(\BA)$-action.

\medskip

The key features of the operator $L$ are as follows:

\medskip

\noindent(a) The operator $L$ can be expressed as an alternating sum of the operators
$$\Eis^G_{P^-}\circ \Upsilon\circ \on{CT}^G_P,$$
where $\Eis^G_P$ and $\on{CT}^G_P$ are the Eisenstein and Constant Term operators, and $\Upsilon$
is a certain intertwining operator. 

\medskip

\noindent(b) When one considers non-ramified functions, the operator $L$ is given by a function on 
$$G(\BO)\backslash G(\BA)/G(K)\times G(\BO)\backslash G(\BA)/G(K)$$
that equals the trace of the Frobenius of the sheaf \eqref{e:* sheaf}. I.e., the functor $\on{Ps-Id}_{\Bun_G}$ 
and the operator $L$ match up via the sheaf-function correspondence. 

\sssec{}

The three situations described above are formally related as follows:

\medskip

The papers \cite{DW,Wa} can be thought of as being a global counterpart for \cite{BBK}. The paper \cite{Ga2}
is un upgrade of \cite{DW,Wa} to a categorical level. 

\medskip

In the present paper we develop an analog of the Deligne-Lusztig functor for the category of $(\fg,K)$-modules. 
The connection to the papers mentioned above is as follows:

\medskip

On the one hand, we regard the category of $(\fg,K)$-modules as an archimedian countrepart of $G(\bK)\mod$.
On the other hand, when we interpret $(\fg,K)$-modules through the localization equivalence, 
this category exhibits many features parallel to $\Shv(\Bun_G)$. 
And finally, it should play a role in the generalization of \cite{Wa} to the case of number
fields. 

\ssec{The present work}

\sssec{}

The object of study of the present paper is the (derived) category of $(\fg,K)$-modules for a symmetric pair $(G,\theta)$,
with a given central character $\chi$. We denote this category $\fg\mod^K_\chi$.

\medskip

We introduce an endo-functor of 
$$\PsId_{\fg\mod^K_\chi}:\fg\mod^K_\chi\to \fg\mod^K_\chi,$$
which is a direct analog of the functor $\PsId_{\Bun_G}$.
In fact, when we interpret $\fg\mod^K_\chi$ via the localization equivalence as the category of twisted D-modules on
$K\backslash X$ (here $X$ is the flag variety of $G$), the functor $\PsId_{\fg\mod^K_\chi}$ corresponds to the
functor $\PsId_{K\backslash X}$ for the stack $K\backslash X$, see \secref{sss:intro mirac}. 

\medskip

We consider $\PsId_{\fg\mod^K_\chi}$ as an analog of the Deligne-Lusztig functor in the context of $(\fg,K)$-modules.
The main results of this paper establish (or conjecture) various properties of $\PsId_{\fg\mod^K_\chi}$ that support
this analogy.

\sssec{}

Here is the summary of our main results:

\medskip

\noindent--We show that the functor $\PsId_{\fg\mod^K_\chi}$ is a self-equivalence of $\fg\mod^K_\chi$, which is the
inverse of the Serre endofunctor $\Se_{\fg\mod^K_\chi}$ (this is \thmref{t:gH miraculous}).  

\medskip

\noindent--In fact, we give a general criterion, when for a DG category $\bC$, the functors $\PsId_\bC$ and $\Se_\bC$
are mutually inverse equivalences (this is \corref{c:main}). 

\medskip

\noindent--We show that $\PsId_{\fg\mod^K_\chi}$ is canonically isomorphic (up to a certain twist) to the composition 
$$\BD^{\on{can}}\circ \BD^{\on{contr}},$$
where $\BD^{\on{can}}$ is the cohomological duality of $\fg\mod^K_\chi$ (which is a direct analog of \eqref{e:coh dual}
for $(\fg,K)$-modules), and $\BD^{\on{contr}}$ is the extension to the derived category of the usual contragredient
duality functor (this is \thmref{t:descr contr}).

\medskip

\noindent--We propose a certain conjecture, which we regard as the analog for $(\fg,K)$-modules of Bernstein's
``2nd adjointness" theorem. We show that this conjecture is equivalent to 
an analog for $(\fg,K)$-modules of the commutation of the diagrams \eqref{e:DL and ind} and \eqref{e:funct equation}.

\medskip

\noindent--We run a plausibility test on our ``2nd adjointness" conjecture, and show that at the level of abelian categories
it reproduces a result of A.~W.~Casselman, D.~Milicic, H.~Hecht and W.~Schmid on the behavior of asymptotics of representations under the 
contragredient duality operation.

\sssec{}

Here are some directions that we do \emph{not} pursue in this paper, but which seem attractive:

\medskip

\noindent--One would like to express the functor $\PsId_{\fg\mod^K_\chi}$ (or its geometric counterpart  
$\PsId_{K\backslash X}$) in a way similar to \eqref{e:DL}, i.e., as a 
complex whose terms are compositions of the Casselman-Jacquet functors and induction functors for 
$\theta$-compatible parabolics in $G$. (This necessitates generalizing the results of \cite{Kim} to the case
when instead of the minimal $\theta$-compatible parabolic, we consider an arbitrary parabolic.) 

\medskip

\noindent--One would like to find an expression for $\PsId_{\fg\mod^K_\chi}$ (or its geometric counterpart  
$\PsId_{K\backslash X}$) via the wonderful compactification $\ol{G}$.

\medskip

\noindent--One would like to generalize the constructions of \cite{Wa} to the number field case, and find 
the relation between his operator $L$ and our functor $\PsId_{\fg\mod^K_\chi}$ at the archimedian primes
(at non-archimedian primes, the ingredients of \cite{BBK} play a role in J.~Wang's constructions). 

\ssec{What is actually done in this paper}

\sssec{}

The main body of this paper begins with \secref{s:Se and Ps}, where we discuss the general formalism 
of Serre and pseudo-identity \emph{operations} and \emph{functors}.

\medskip

For a pair of DG categories $\bC$ and $\bD$, the Serre and the pseudo-identity \emph{operations} are both contravariant  functors
$$\Se,\Ps: \on{Funct}_{\on{cont}}(\bC,\bD)\to \on{Funct}_{\on{cont}}(\bD,\bC),$$
which take colimits to limits.

\medskip

The Serre and pseudo-identity functors are defined by
$$\Se_\bC:=\Se(\on{Id}_\bC) \text{ and } \PsId_\bC:=\Ps(\on{Id}_\bC).$$

\medskip

The main results of this section are \propref{p:main} and \corref{c:main}. The former says that if certain
finiteness conditions are satisfied (preservation of compactness), for a continuous functor $F:\bC\to \bD$ we 
have
$$\Se_\bC\circ \Ps(F)\circ \Se_\bD\simeq \Se(F).$$

The latter says that if $\bC$ satisfies a certain finiteness condition, then the functors $\Se_\bC$ and $\PsId_\bC$
are mutually inverse equivalences.

\sssec{}

In \secref{s:ex} we show that some DG categories that naturally arise in geometric representation theory
satisfy the assumption of \corref{c:main} mentioned above; in particular, for
such categories the Serre and the pseudo-identity functors are mutually inverse equivalences.

\medskip

One set of examples consists of categories of (twisted) D-modules on algebraic stacks $\CY$ that have
finitely many isomorphism classes of points. For example, a stack of the form $\CY:=H\backslash Y$,
where $H$ is an algebraic group acting on a scheme $Y$ with finitely many orbits has this property.

\medskip

We also consider a variant, where we have a $T$-torsor $\wt\CY\to \CY$ (where $T$ is a torus), and we consider 
$\lambda$-monodromic D-modules on $\wt\CY$ for a character $\lambda\in \ft^*$. 

\medskip

Another set of examples comes from representations of Lie algebras. Let $G$ be a reductive group with Lie
algebra $\fg$; fix a character $\chi$ of $Z(\fg):=Z(U(\fg))$. We show that if $H\subset G$ is spherical
(i.e., has finitely many orbits on the flag variety $X$ of $G$), then the corresponding category
$$\fg\mod^H_\chi$$
(i.e., the derived version of the category of $(\fg,H)$-modules with the fixed central character $\chi$) 
satisfies the assumptions of \corref{c:main}. In particular, the Serre and the pseudo-identity functors are mutually 
inverse equivalences for $\fg\mod^H_\chi$. 

\medskip

We also establish a variant of this result, where instead of $\fg\mod^H_\chi$, we consider $\fg\mod^H_{\{\chi\}}$,
i.e., we let our modules have a \emph{generalized} central character $\chi$.  

\medskip

As a byproduct (and using a result from \cite{Kim}) we reprove the result from \cite{BBM} that says that the
Serre functor on category $\CO$ (or, equivalently, on the category of $N$-equivariant twisted D-modules
on the flag variety) is given by the square of long intertwining functor.  

\sssec{}

In \secref{s:contr} we study the functors obtained by composing the canonical (=cohomological) duality\footnote{In the formula 
below and elsewhere, the notation $\bC^c$ means the subcategory of compact objects
in a given DG category $\bC$.}

$$\BD^{\on{can}}_{\fg,H}:(\fg\mod_\chi^H)^c\to (\fg\mod_{-\chi}^H)^c$$
with the Serre functor on $(\fg\mod_{-\chi}^H)^c$ (and also a twist by a certain 
determinant line).

\medskip

We consider the following cases:

\medskip

\noindent(i) $H=G$, so that $\fg\mod_\chi^H$ is $\Rep(G)$, the category of algebraic representations of $G$; 

\medskip

\noindent(ii) $H=N$, the unipotent radical of the Borel, so that $\fg\mod_\chi^H$ is category $\CO$; 

\medskip

\noindent(iii) $H=K$, where $K=G^\theta$ in the case of a symmetric pair, so that $\fg\mod_\chi^H$
is the derived category of $(\fg,K)$-modules; 

\medskip

\noindent(iv) $H=M_K\cdot N$, also in the case of a symmetric pair, where $N$ is now the unipotent radical of a $\theta$-minimal
parabolic $P$, and $M_K=K\cap P$.

\medskip

We show that, after ind-extension, in cases (i) and (iii) above, the resulting functor 
$$\fg\mod_\chi^H \to (\fg\mod_{-\chi}^H)^{\on{op}}$$ 
is a derived version of the usual contragredient duality functor. 

\medskip

In cases (ii) and (iv), the same happens after we compose with the corresponding long intertwining functor
$$\fg\mod_{-\chi}^N\to \fg\mod_{-\chi}^{N^-} \text{ and } \fg\mod_{-\chi}^{M_K\cdot N}\to \fg\mod_{-\chi}^{M_K\cdot N^-}$$
(note that in these cases, contragredient duality naturally replaces the the local finiteness condition with respect to $N$
by that with respect to $N^-$). 
 
\ssec{Principal series and the ``2nd adjointness" conjecture}
 
\sssec{}

In \secref{s:2nd adj} we study the two (mutually Verdier conjugate) principal series functors
$$\Av^{K/M_K}_! \text{ and } \Av^{K/M_K}_*$$
that map
$$\fg\mod_\chi ^{M_K\cdot N}\to \fg\mod_\chi^K,$$
and their various adjoints. The above two functors are (loose) analogs of the two Eisenstein series functors
$\Eis^G_P$ and $\wt\Eis^G_P$ in \eqref{e:funct equation}.

\medskip

The functor $\Av^{K/M_K}_!$ is naturally the \emph{left} adjoint of the functor
$$\Av^N_*:\fg\mod_\chi^K\to \fg\mod_\chi ^{M_K\cdot N},$$
and the functor $\Av^{K/M_K}_*$ is naturally the \emph{right} adjoint of the functor
$$\Av^N_!:\fg\mod_\chi^K\to \fg\mod_\chi^{M_K\cdot N}.$$

\sssec{}

We recall the ``2nd adjointness" conjecture from \cite{Kim}, which says that the functor
$\Av^{K / M_K}_*$ is also the \emph{left} adjoint (up to a certain determinant line) of the functor
$$\Av^N_*\circ \Av^{N^-}_!:\fg\mod_\chi^K\to \fg\mod_\chi^{M_K\cdot N}.$$

The functor $\Av^N_*\circ \Av^{N^-}_!$ that appears above is known as the \emph{Casselman-Jacquet}
functor; we denote it by $J^-$.

\medskip

An equivalent formulation of the ``2nd adjointness" conjecture is that there exists a canonical
isomorphism between
$$\Av^{K/M_K}_* \circ  \Av^N_!\text{ and } \Av^{K/M_K}_!$$
as functors 
$$\fg\mod_\chi^{M_K\cdot N^-}\rightrightarrows \fg\mod_\chi^K$$
(up to a certain determinant line). 

\sssec{}

We also prove that there exists a canonical isomorphism (up to a cohomological shift) between 
$$\PsId_{\fg\mod_\chi^K}\circ \Av^{K/M_K}_* \text{ and } \Av^{K/M_K}_!$$
as functors $\fg\mod_\chi^{M_K\cdot N^-}\rightrightarrows \fg\mod_\chi^K$. 

\medskip

Juxtaposing, we obtain that the ``2nd adjointness" conjecture is equivalent to the fact that 
the following diagram commutes (up to a certain determinant line and a cohomological shift):
$$
\CD
\fg\mod_\chi^K  @<{\Av^{K/M_K}_*}<< \fg\mod_\chi^{M_K\cdot N} \\
@V{\on{Ps-Id}_{\fg\mod_\chi^K}}VV   @VV{\Av^{N^-}_!}V  \\
\fg\mod_\chi^K  @<{\Av^{K/M_K}_*}<< \fg\mod_\chi^{M_K\cdot N^-}.
\endCD
$$

We consider this to be a (loose) analog of the commutative diagram \eqref{e:funct equation}. 

\sssec{}

Finally, we show that our ``2nd adjointness" conjecture is equivalent to an isomorphism of functors
$$\BD_{\fg,\chi}^{\on{contr}}\circ J\simeq J^-\circ \BD_{\fg,\chi}^{\on{contr}}:
(\fg\mod_\chi^K)^c\rightrightarrows  (\fg\mod_{-\chi}^{M_K\cdot N})^c,$$
where $J$ is the counterpart of $J^-$ where we swap $P$ for $P^-$. The latter isomorphism is known
at the level of abelian categories, for $k = \mathbb{C}$, due to Casselman, Milicic and later Hecht and Schmid (see \cite{Ca}, \cite{M}, \cite{HS}). Their approach is anayltic, using asymptotics of matrix coefficients.

\ssec{Organization of the paper}

\sssec{}

In \secref{s:Se and Ps} we discuss the general formalism of Serre and pseudo-identity operations and functors for DG categories. 

\sssec{}

In \secref{s:ex} we consider examples of DG categories that come from geometry and representation theory, and show
that for some of these categories, the Serre and pseudo-identity functors are mutually inverse. 

\sssec{}

In \secref{s:contr} we relate the Serre functor on certain representation-theoretic categories to the functor of
contragredient duality.

\sssec{}

In \secref{s:2nd adj} we relate the material of the previous sections of this paper to an analog of the ``2nd adjointness"
conjecture for $(\fg,K)$-modules. 

\ssec{Notation and conventions}

The conventions in this paper follow those of \cite{Kim}. 

\sssec{}

Throughout the paper we will working over a ground field $k$, assumed algebraically closed and of characteristic $0$.
We let $G$ be a connected reductive algebraic group over $k$, and we let $X$ denote the flag variety of $G$. 

\sssec{}  

We will be working with DG categories over $k$ (see \cite[Chapter 1, Sect. 10]{GR2} for a concise summary of DG categories).
All functors between DG categories are assumed to be exact, i.e., preserving finite limits (equivalently, colimits or cones). 

\medskip

All DG categories will be assumed \emph{cocomplete} (i.e., contain infinite direct sums). Unless specified otherwise, 
when discussing a functor between two DG categories, we will assume that this functor is \emph{continuous}, i.e.,
preserves infinite direct sums (equivalently, colimits or filtered colimits). 

\medskip

We denote by $\Vect$ the DG category of complexes of vector spaces. For a DG category $\bC$ and $\bc,\bc'\in \bC$,
we let $\CHom_\bC(\bc,\bc')\in \Vect$ denote the corresponding Hom complex.

\sssec{} 

In this paper, we will consider the operation of \emph{tensor product} of DG categories,
$$\bC_1,\bC_2\mapsto \bC_1\otimes \bC_2,$$
\cite[Chapter 1, Sect.10.4]{GR2}.

\medskip

This operation is functorial with respect to \emph{continuous} functors
$\bC_1\to \bC'_1$ and $\bC_2\to \bC'_2$. It makes the $\infty$-category of cocomplete DG categories 
and continuous functors into a symmetric monoidal $\infty$-category; the unit is given by 
DG category $\Vect$.

\medskip

In order to have tensor product, one does need to work with DG categories, rather than triangulated
categories. This is why the usage of higher algebra is unavoidable for this paper (unlike its
predecessor \cite{Kim}).  

\sssec{}  \label{sss:DG categ} 

In any symmetric monoidal $\infty$-category, given an object one can ask for its dualizability. This way we arrive
at the notion of \emph{dualzable} DG category.

\medskip

A duality data for a DG category $\bC$ is another DG category $\bC^\vee$ and a pair of functors
$$\langle-,-\rangle_\bC:\bC\otimes \bC^\vee\to \Vect$$
and 
$$\Vect\to \bC^\vee\otimes \bC$$
that satisfy the appropriate axioms. 

\medskip

It is known that every compactly generated category $\bC$ is dualizable. In this case $\bC^\vee$ is also
compactly generated and we have a canonical equivalence
$$(\bC^\vee)^c\simeq (\bC^c)^{\on{op}},$$
see \cite[Chapter 1, Proposition 7.2.3]{GR2}.

\medskip

Explicitly, the functor $(\bC^c)^{\on{op}}\to \bC^\vee$ is characterized by the requirement that the composition
$$\bC\times (\bC^c)^{\on{op}}\to \bC\times \bC^\vee\to \bC\otimes \bC^\vee \overset{\langle-,-\rangle_\bC}\longrightarrow \Vect$$
is given by
$$\bc,\bc'\mapsto \CHom_\bC(\bc',\bc).$$

\sssec{Derived algebraic geometry}

This paper will make a mild use of derived algebraic geometry; see \cite[Chapters 2 and 3]{GR2} for a brief summary,
in particular for our usage of the notation $\QCoh(-)$.

\medskip

All (derived) schemes will be assume laft (locally almost of finite type); see \cite[Chapters 2, Sect.1.7]{GR2} for
what this means. 

\sssec{D-modules}

Given a scheme/algebraic stack $\CY$, we will denote by $\Dmod(\CY)$ the DG category of D-modules on $\CY$;
see \cite{GR1}. 

\medskip

Given a twisting $\lambda$ on $\CY$, we will denote by $\Dmod_\lambda(\CY)$ the corresponding DG category
of twisted D-modules, see \cite[Sects. 6 and 7]{GR1}. 

\ssec{Acknowledgements}

The research of D.G. is supported by NSF grant DMS-1063470. A.Y.D. would like to thank Sam Raskin for very helpful discussions on higher algebra.

\section{The Serre and pseudo-identity functors} \label{s:Se and Ps}

In this section all categories will be assumed compactly generated. 

\ssec{The Serre operation}

\sssec{}

We introduce some terminology:

\medskip

We shall call a continuous functor $F:\bC\to \bD$ \emph{proper} if it maps compact objects to compact objects.
(Equivalently, if $F$ admits a \emph{continuous} right adjoint.)

\medskip

We shall say that $\bC$ is proper if the evaluation functor
\begin{equation} \label{e:ev} 
\bC\otimes \bC^\vee\to \Vect
\end{equation}
is proper. Equivalently, if for every $\bc,\bc'\in \bC^c$, the object $\CHom_\bC(\bc,\bc')\in \Vect$
is compact. 

\medskip

We introduce the dualization functor
$$\BD_\bC:\bC\to (\bC^\vee)^{\on{op}}$$
be requiring that it be the ind-extension of the tautological functor
$$\bC^c \simeq ((\bC^\vee)^c)^{\on{op}}\to  (\bC^\vee)^{\on{op}}.$$

\medskip

In other words, the functor $\BD_\bC$, when viewed as a \emph{contravariant} functor $\bC\to \bC^\vee$
sends colimits to limits and is the tautological functor on $\bC^c$.

\medskip

We shall say that $\bc\in \bC$ is \emph{reflexive} if the tautological map
$$\bc\to \BD_{\bC^\vee}\circ \BD_\bC(\bc)$$
is an isomorphism.

\medskip

For example, for $\bC=\Vect$, the functor $\BD_\Vect$ is the usual dualization functor 
$$V\mapsto V^*,$$
and $V\in \Vect$ is reflexive if and only if it has finite-dimensional cohomologies (but it may have
infinitely many non-vanishing cohomology groups). 

\medskip

We shall say that $\bC$ is \emph{reflexive} if the evaluation functor \eqref{e:ev} sends compact objects
to reflexive objects. Equivalently, if for every $\bc,\bc'\in \bC^c$, the object $\CHom_\bC(\bc,\bc')\in \Vect$
has finite-dimensional cohomologies. 

\sssec{}

Let $\bC$ and $\bD$ be DG categories. We define the Serre operation
$$\Se:\Funct_{\on{cont}}(\bC,\bD)\to \Funct_{\on{cont}}(\bD,\bC)^{\on{op}}$$
as follows.

\medskip

Namely, for $F\in \Funct_{\on{cont}}(\bC,\bD)$, $\bc\in \bC$ and $\bd\in \bD^c$ we set
$$\CHom_\bC(\bc,\Se(F)(\bd)):=\CHom_\bD(\bd,F(\bc))^*.$$

\sssec{}

First, we observe:

\begin{lem}
The functor $\Se$ preserves colimits. 
I.e., when viewed as a contravariant functor, $\Se$ takes colimits in 
$\Funct_{\on{cont}}(\bC,\bD)$ to limits in $\Funct_{\on{cont}}(\bD,\bC)$.
\end{lem}

\begin{proof}
For a colimit diagram 
$$(i\in I) \mapsto F_i, \quad \underset{i\in I}{\on{colim}}\,F_i =F,$$
in $\Funct_{\on{cont}}(\bC,\bD)$, and $\bc\in \bC$, $\bd\in \bD^c$, we have
\begin{multline*}
\CHom(\bc, (\underset{i\in I}{\on{lim}}\, \Se(F_i))(\bd)) 
=\underset{i\in I}{\on{lim}}\, \CHom(\bc, \Se(F_i)(\bd)) = \underset{i\in I}{\on{lim}}\, \CHom(\bd,F_i(\bc))^* =\\
= \left(\underset{i\in I}{\on{colim}}\, \CHom(\bd,F_i(\bc))\right){}^* =\CHom(\bd,F(\bc))^*=
\CHom(\bc, \Se(F)(\bd)),
\end{multline*}
as desired.

\end{proof}

\ssec{The Serre functor}  \label{ss:Se}

\sssec{}

We introduce the Serre endofunctor of $\bC$, denoted $\Se_\bC$, by setting
$$\Se_\bC:=\Se(\on{Id}_\bC).$$

We shall say that $\bC$ is \emph{Serre} if $\Se_\bC$ is an equivalence. 

\sssec{}

From the definitions we obtain:

\begin{lem} $(\Se_\bC)^\vee\simeq \Se_{\bC^\vee}$ as endofunctors of $\bC^\vee$. 
\end{lem}

\begin{cor}
$\bC$ is Serre if and only $\bC^\vee$ is.
\end{cor} 

\sssec{}

We now claim:

\begin{prop}
The functor $\Se_\bC|_{\bC^c}$ is fully faithful if and only if $\bC$ is reflexive.  
\end{prop}

\begin{proof}

For $\bc,\bc'\in \bC^c$, we have:
$$\CHom_\bC(\Se_\bC(\bc),\Se_\bC(\bc')) =\CHom_\bC(\bc',\Se_\bC(\bc))^*
=\CHom_\bC(\bc,\bc')^{**},$$
where the map
$$\CHom_\bC(\bc,\bc')\to \CHom_\bC(\Se_\bC(\bc),\Se_\bC(\bc'))$$
is the canonical map
$$\CHom_\bC(\bc,\bc')\to \CHom_\bC(\bc,\bc')^{**}.$$

\end{proof}

\begin{cor} 
Assume that $\bC$ is reflexive. \hfill

\smallskip

\noindent{\em(a)} If $\Se_\bC$ is proper, then it is fully faithful, and 
the right adjoint to $\Se_\bC$ provides a continuous left inverse of $\Se_\bC$.

\smallskip

\noindent{\em(b)} If $\Se_{\bC^\vee}$ is proper, then $\Se_\bC$ is fully faithful, and admits a left adjoint, which is also a 
right inverse of $\Se_\bC$.

\end{cor} 

\begin{cor}\label{c:serre}
Assume that $\bC$ is reflexive and that $\Se_\bC$ and $\Se_{\bC^\vee}$ are both proper. Then $\bC$ is Serre. 
\end{cor}

 



\ssec{Some examples}

\sssec{}

Let $Y$ be an eventually coconnective derived scheme, and $\bC=\QCoh(Y)$. 
It is easy to see that $\QCoh(Y)$ is reflexive if and only if $Y$ is proper, which we will from now on assume. 

\medskip

Then the Serre functor on $Y$ is given by $\CF\mapsto \CF\otimes \omega_Y$, where $\omega_Y$ is the dualizing
object on $Y$. 

\medskip

From here it is clear that $\QCoh(\CY)$ is Serre if and only $Y$ is Gorenstein, which by definition means that $\omega_Y$
is a shifted line bundle. 

\sssec{}

Let $Y$ be a smooth scheme and let $y\in Y$ be a point. Consider the category $\bC=\QCoh(Y)_{\{y\}}$, which is the
full subcategory of $\QCoh(Y)$ that consists of objects set-theoretically supported at $y$. 

\medskip

It is easy to see that $\QCoh(Y)_{\{y\}}$ is proper, and the Serre functor on it is given by
$$\CF\mapsto \CF\otimes \omega_Y.$$

\medskip

The skyscraper $k_y\in \QCoh(Y)_{\{y\}}$ defines a functor $(i_y)_*:\Vect\to \QCoh(Y)_{\{y\}}$. Note that we have
$$\Se_{\QCoh(Y)_{\{y\}}}\circ (i_y)_*\simeq (i_y)_*\circ \Se_{\Vect}\circ (- \otimes \omega_{Y,y}),$$
where $\omega_{Y,y}$ is the fiber of $\omega_Y$ at $y$, which equals $\Lambda^{\dim(Y)}(T^*_y(Y))[\dim(Y)]$.

\sssec{}

Let $N$ be a unipotent algebraic group, and consider the category $\bC=\Rep(N)$. This category is proper, and we
claim the Serre functor on it is given by
$$\CM\mapsto \CM\otimes \ell_N,$$
where $\ell_N$ is the graded line $\Lambda^{\dim(\fn)}(\fn)[\dim(\fn)]$ (in particular, $\Rep(N)$ is Serre).

\medskip

Indeed, this follows from the fact that for $\CM_1,\CM_2\in \Rep(N)^c$, we have
$$\CHom_{\Rep(N)}(\CM_1,\CM_2)\simeq \on{C}^\bullet(\fn,\CM_1^\vee\otimes \CM_2),$$
and hence
\begin{multline*}
\CHom_{\Rep(N)}(\CM_1,\Se_{\Rep(N)}(\CM_2))=
\CHom_{\Rep(N)}(\CM_2,\CM_1)^*\simeq \on{C}^\bullet(\fn,\CM_2^\vee\otimes \CM_1)^*\simeq \\
\simeq \on{C}_\bullet(\fn,\CM_1^\vee\otimes \CM_2)\simeq \on{C}^\bullet(\fn,\CM_1^\vee\otimes \CM_2)\otimes \ell_N
\simeq \CHom_{\Rep(N)}(\CM_1,\CM_2)\otimes \ell_N,
\end{multline*}
as required.

\sssec{}

Let $K$ be a reductive group. Then the category $\Rep(K)$ is proper, and it is easy to see that
$\Se_{\Rep(K)}$ is canonically isomorphic to the identity functor.

\ssec{The pseudo-identity functor}

\sssec{}

We define the functor
$$\on{Ps}:\Funct_{\on{cont}}(\bC,\bD)\to \Funct_{\on{cont}}(\bD,\bC)^{\on{op}}$$
to be the functor $\BD_{\bC^\vee\otimes \bD}$, where we identify
$$\Funct_{\on{cont}}(\bC,\bD)\simeq \bC^\vee\otimes \bD \text{ and }
\Funct_{\on{cont}}(\bD,\bC)\simeq \bD^\vee\otimes \bC\simeq (\bC^\vee\otimes \bD)^\vee.$$

\medskip

By construction, $\on{Ps}$ preserves colimits. I.e., when viewed as a contravariant functor
$$\Funct_{\on{cont}}(\bC,\bD)\to \Funct_{\on{cont}}(\bD,\bC)$$
it takes colimits to limits.

\sssec{}

We define the endofunctor $\PsId_\bC$ of $\bC$ by 
$$\PsId_\bC:=\Ps(\Id_\bC).$$

We shall say that $\bC$ is \emph{Gorenstein} if the functor $\PsId_\bC$ is an equivalence
(see \cite[Sect. 5.4]{Ga1}, where the origin of the terminology is explained). 

\sssec{}

The following is tautological from the definitions:

\begin{lem} $\PsId_{\bC^\vee}\simeq (\PsId_\bC)^\vee$ as endo-functors of $\bC^\vee$.
\end{lem}

\ssec{Relationship between the Serre and pseudo-identity functors}

\sssec{}

The following observation will play a key role in this paper: 

\begin{prop} \hfill    \label{p:main}

\smallskip

\noindent{\em(a)}
There is natural transformation 
$$\{F\mapsto \Se_\bC\circ \Ps(F)\circ \Se_\bD\}\, \Rightarrow \{F\mapsto \Se(F)\},$$
when both sides are viewed as \emph{contravariant} functors 
$\Funct_{\on{cont}}(\bC,\bD)\to \Funct_{\on{cont}}(\bD,\bC)$.

\smallskip

\noindent{\em(b)} Assume that $\Se_{\bC^{\vee}}$ and $\Se_\bD$ are proper, and suppose that
for $\bc,\bc' \in \bC^c$ and $\bd,\bd'\in \bD^c$, the map 
$$\CHom_\bC(\bc',\bc)^*\otimes \CHom_\bD(\bd,\bd')^*\to \left(\CHom_\bD(\bd,\bd')\otimes \CHom_\bC(\bc',\bc)\right)^*$$
is an isomorphism. Then the above natural transformation is an isomorphism.

\end{prop}

\begin{rem}  \label{r:main}
Note that the second condition in (b) is satisfied when:

\begin{itemize}

\item Either $\bC$ or $\bD$ is proper.

\item Either $\bC$ or $\bD$ is reflexive \emph{and} for $\bc,\bc'\on \bC^c$ and $\bd,\bd'\in \bD^c$, the objects
$$\CHom_\bC(\bc',\bc)\in \Vect \text{ and } \CHom_\bD(\bd,\bd') \in \Vect$$
are either both eventually connective or connective (i.e., either both are in $\Vect^{>-\infty}$ or both are in $\Vect^{<\infty}$).

\end{itemize} 

\end{rem} 

\begin{proof}

Since $\Se$ sends colimits to limits, in order to construct the natural transformation in question, it is enough to do so after precomposition with the functor $$ (\bC^c)^{\on{op}}\times \bD^c \to \Funct_{\on{cont}} (\bC , \bD).$$

\medskip

For $\bc'\in \bC^c$, $\bd'\in \bD^c$, the corresponding functor $F:\bC\to \bD$ is given by
$$F(\bc)=\bd'\otimes \CHom_\bC(\bc',\bc),$$ 
the functor $\Ps(F)$ is given by
$$\Ps(F)(\bd)=\bc'\otimes \CHom_\bD(\bd',\bd).$$

Hence, the functor 
$G_1:=\Se_\bC\circ \Ps(F)\circ \Se_\bD$
is determined by
\begin{multline*}
\CHom_\bC(\bc,G_1(\bd))=\CHom_\bC(\bc,\Se_\bC(\bc'))\otimes \CHom_\bD(\bd',\Se_\bD(\bd))\simeq \\
\simeq \CHom_\bC(\bc',\bc)^*\otimes \CHom_\bD(\bd,\bd')^*, \quad \bc\in \bC^c,\bd\in \bD^c.
\end{multline*}

\medskip

The functor $G_2:=\Se(F)$
is determined by
\begin{multline*}
\CHom_\bC(\bc,G_2(\bd))=\CHom_\bD(\bd,\bd'\otimes \CHom_\bC(\bc',\bc))^*\simeq \\
\simeq \left(\CHom_\bD(\bd,\bd')\otimes \CHom_\bC(\bc',\bc)\right)^*, \quad \bc\in \bC^c,\bd\in \bD^c.
\end{multline*}

\medskip

The required natural transformation is now given by
\begin{equation} \label{e:nat trans 2} 
\CHom_\bC(\bc',\bc)^*\otimes \CHom_\bD(\bd,\bd')^*\to \left(\CHom_\bD(\bd,\bd')\otimes \CHom_\bC(\bc',\bc)\right)^*.
\end{equation}

\medskip

Let us now assume that \eqref{e:nat trans 2} is an isomorphism and that $\Se_{\bC^{\vee}}$ and $\Se_\bD$ are proper. 
Let us show that the natural transformation
$$\Se_\bC\circ \Ps(F)\circ \Se_\bD\to \Se(F)$$
is an isomorphism for any $F\in \Funct_{\on{cont}}(\bC,\bD)$. For that it suffices to show that the functor
$$F\mapsto \Se_\bC\circ \Ps(F)\circ \Se_\bD$$
takes colimits in $\Funct_{\on{cont}}(\bC,\bD)$ to limits in $\Funct_{\on{cont}}(\bD,\bC)$. 

\medskip

It suffices to show that composition with
$\Se_\bC$ and pre-composition with $\Se_\bD$ preserve limits. This follows from the following lemma:

\begin{lem}

Let $\bC,\bD,\bE$ be compactly generated categories.

\smallskip

\noindent{\em(a)} If $F \in \Funct_{\on{cont}} (\bE , \bC)$ admits a continuous right adjoint (equivalently, $F$ is proper), then the functor $\Funct_{\on{cont}}(\bC,\bD) \to \Funct_{\on{cont}}(\bE,\bD)$ given by precomposition with $F$ preserves limits.

\smallskip

\noindent{\em(b)} If $F \in \Funct_{\on{cont}} (\bD , \bE)$ admits a left adjoint (equivalently, $F^{\vee}$ is proper), then the functor $\Funct_{\on{cont}}(\bC,\bD) \to \Funct_{\on{cont}}(\bC,\bE)$ given by postcomposition with $F$ preserves limits.

\end{lem}

\end{proof}

\sssec{}

Summarizing, we obtain the following result that we will use extensively:

\begin{cor}   \label{c:main}
Let $\bC$ be reflexive, and such that for $\bc,\bc' , \bc_1 , \bc'_1 \in \bC^c$ the map 
$$\CHom_\bC(\bc',\bc)^*\otimes \CHom_\bC(\bc_1,\bc'_1)^*\to \left( \CHom_\bC(\bc_1,\bc'_1) \otimes \CHom_\bC(\bc',\bc) \right)^*$$
is an isomorphism. Assume also that $\Se_\bC$ and $\Se_{\bC^\vee}$ are proper. 
Then the functors $\Se_\bC$ and $\PsId_\bC$ are mutually inverse. In particular,
$\bC$ is Serre and Gorenstein. 
\end{cor} 

\begin{proof}
	By \propref{p:main}, we have an isomorphism $\Se_{\bC} \circ \PsId_{\bC} \circ \Se_{\bC} \simeq \Se_{\bC}$. By \corref{c:serre}, The functor $\Se_{\bC}$ is an equivalence.
\end{proof}

\begin{rem} \label{r:main main}
By Remark \ref{r:main}, the first two conditions in the corollary is satisfied if $\bC$ is proper. 

\medskip

More generally, it is satisfied if $\bC$ is 
reflexive and for $\bc,\bc'\in \bC^c$, the object  
$\CHom_\bC(\bc',\bc)\in \Vect$ is eventually coconnective, i.e., lies in $\Vect^{>-\infty}$. 

\medskip

For example, the latter happens if $\bC$
carries a t-structure and all compact objects in $\bC$ are bounded, i.e., lie in $\bC^{\geq n_1,\leq n_2}$ for some $n_1\leq n_2$. 
\end{rem}

\section{Serre and Gorenstein categories in geometric representation theory} \label{s:ex}

\ssec{Examples arising from D-module categories}

\sssec{}

Let $\CY$ be a QCA algebraic stack (see \cite[Definition 1.1.8]{DrGa1} for what this means), 
and let $\lambda$ be a twisting on $\CY$ (see \cite[Sect. 6]{GR1}). Consider the category $$\bC:=\Dmod_\lambda(\CY).$$

\medskip

Recall that the dual category $\Dmod_\lambda(\CY)^\vee$ identifies canonically with $\Dmod_{-\lambda}(\CY)$.
Under this identification, the evaluation map
\begin{equation} \label{e:Dmod pairing}
\Dmod_\lambda(\CY)\otimes \Dmod_{-\lambda}(\CY)\to \Vect
\end{equation}
is given by
$$\Dmod_\lambda(\CY)\otimes \Dmod_{-\lambda}(\CY)\simeq \Dmod_{\lambda,-\lambda}(\CY\times \CY)
\overset{\Delta^!_\CY}\longrightarrow \Dmod(\CY) \overset{(p_\CY)_*}\longrightarrow \Vect,$$
where $p_\CY$ is the projection $\CY\to \on{pt}$, and for a morphism $f$ we denote by $f_*$ the
\emph{renormalized} pushforward of \cite[Sect. 9.3]{DrGa1}.

\medskip

The unit map is given by
$$\Vect\overset{k\mapsto \omega_\CY}\longrightarrow \Dmod(\CY) \overset{(\Delta_\CY)_*}\longrightarrow 
\Dmod_{\lambda,-\lambda}(\CY\times \CY)\simeq \Dmod_\lambda(\CY)\otimes \Dmod_{-\lambda}(\CY).$$

\medskip

We denote by 
$$\BD^{\on{Verdier}}_\CY:\Dmod_\lambda(\CY)^c\to \Dmod_{-\lambda}(\CY)^c$$
the corresponding contravariant dualization functor, and also its ind-extension
$$\BD^{\on{Verdier}}_\CY:\Dmod_\lambda(\CY)\to (\Dmod_{-\lambda}(\CY))^{\on{op}}.$$

\sssec{}

Thus, objects of 
$$\Dmod_{-\lambda}(\CY)\otimes \Dmod_\lambda(\CY)\simeq \Dmod_{-\lambda,\lambda}(\CY\times \CY)$$
define continuous endofunctors of $\Dmod_\lambda(\CY)$. 

\medskip

Following \cite[Sect. 3.1]{Kim}, we introduce the functor
$$\PsId_\CY:\Dmod_\lambda(\CY)\to \Dmod_\lambda(\CY)$$
to be given by the object
$$(\Delta_\CY)_!(k_\CY)\in \Dmod_{-\lambda,\lambda}(\CY\times \CY),$$
where 
$$k_\CY:=\BD^{\on{Verdier}}_\CY(\omega_\CY).$$

\medskip

The following is tautological from the definitions:

\begin{lem}
The functor $\PsId_\CY$ identifies with the functor $\PsId_{\Dmod_\lambda (\CY)}$. 
\end{lem} 

In other words, the functor $\PsId_\bC$ is the abstract version of the geometrically defined functor $\PsId_\CY$. 

\sssec{}

We will prove:

\begin{thm}  \label{t:Y miraculous}
Assume that $\CY$ has a finite number of isomorphism classes of $k$-valued points. Then $\Dmod_\lambda(\CY)$
is proper, Serre and Gorenstein, and the functors $\Se_{\Dmod_\lambda(\CY)}$ and $\PsId_\CY$ are mutually inverse.
\end{thm} 

\ssec{Proof of \thmref{t:Y miraculous}}

The proof will rely on material from the paper \cite{DrGa1}. 

\sssec{}

We will use the following lemma, proved below:

\begin{lem} \label{l:when compact}
Let $\CY$ have a finite number of isomorphism classes of $k$-valued points. Then an object $\CF\in \Dmod_\lambda(\CY)$
is compact if and only if for every coherent $\CF'\in \Dmod_\lambda(\CY)$, the object
$$\CHom_{\Dmod_\lambda}(\CF',\CF)\in \Vect$$
is compact.
\end{lem}

We now proceed with the proof of \thmref{t:Y miraculous}. We will verify that the conditions of \corref{c:main} hold. 

\sssec{}

First, we show that $\Dmod_\lambda(\CY)$ is proper. Since every compact object of $\Dmod_\lambda(\CY)$ is coherent,
this follows immediately from the ``only if" direction in \lemref{l:when compact}.

\sssec{}

Let us now show that the functors $\Se_{\Dmod_\lambda(\CY)}$ and $\Se_{\Dmod_\lambda(\CY)^\vee}\simeq \Se_{\Dmod_{-\lambda}(\CY)}$ 
are proper. By symmetry, it suffices to consider the former case. 

\medskip

By the ``if" direction in \lemref{l:when compact}, it suffices to show that for $\CF\in \Dmod_\lambda(\CY)^c$
and $\CF'\in \Dmod_\lambda(\CY)^{\on{coh}}$, the object
$$\CHom_{\Dmod_\lambda}(\CF',\Se_{\Dmod_\lambda(\CY)}(\CF))\in \Vect$$
is compact.

\medskip

Since $\CHom_{\Dmod_\lambda}(\CF',\Se_{\Dmod_\lambda(\CY)}(\CF))$ is the dual of 
$\CHom_{\Dmod_\lambda}(\CF,\CF')$, so it suffices to show that the latter is compact.

\medskip

Recall now from \cite[Corollary 8.4.2]{DrGa1} that Verdier duality on $\CY$, which we can regard as a contravariant
equivalence
$$\BD^{\on{Verdier}}_\CY:\Dmod_\lambda(\CY)^c\to \Dmod_{-\lambda}(\CY)^c,$$
extends to a contravariant equivalence
$$\BD^{\on{Verdier}}_\CY:\Dmod_\lambda(\CY)^{\on{coh}}\to \Dmod_{-\lambda}(\CY)^{\on{coh}}.$$

Hence, using
$$\CHom_{\Dmod_\lambda}(\CF,\CF')\simeq \CHom_{\Dmod_\lambda}(\BD^{\on{Verdier}}_\CY(\CF'),\BD^{\on{Verdier}}_\CY(\CF)),$$

the required assertion follows from the ``only if" direction in \lemref{l:when compact}.

\qed[\thmref{t:Y miraculous}]

\sssec{Proof of \lemref{l:when compact}, the ``only if" direction}

By \cite[Proposition 9.2.3]{DrGa1}, compact objects in $\Dmod_\lambda(\CY)$ are characterized by the following two
properties: these are objects that are (i) coherent, and (ii) \emph{safe} (see \cite[Definition 9.2.1]{DrGa1} for what
the term ``safe" refers to).

\medskip

Now, the condition that $\CY$ has a finite number of isomorphism classes of $k$-valued points implies that
all coherent objects in $\Dmod_\lambda(\CY)$ are holonomic. This readily implies that for two
such objects $\CF,\CF'$, 
$$\CHom_{\Dmod_\lambda(\CY)}(\CF',\CF)$$
is finite-dimensional in each degree. 

\medskip

Now, by \cite[Lemma 9.4.4(a)]{DrGa1}, if $\CF'$ is coherent and $\CF$ is safe and bounded, 
then $$\CHom_{\Dmod_\lambda(\CY)}(\CF',\CF)\simeq \Gamma_{\on{dR}}(\CY,\BD_\CY^{\on{Verdier}}(\CF')\sotimes \CF)$$
is concentrated in finitely many cohomological degrees. 

\sssec{Proof of \lemref{l:when compact}, the ``if" direction}

Write $\CY$ is a union of locally closed substacks $\CY_i$, where each
$\CY_i$ has a unique isomorphism class of $k$-valued points, i.e., $\CY_i$ is of the form $\on{pt}/H_i$
for an affine algebraic group $H_i$. Denote by $j_i$ the locally closed embedding $\CY_i\hookrightarrow \CY$.

\medskip

Since every object in $\Dmod_\lambda(\CY_i)^c$ is holonomic, the functor $(j_i)_!$, left adjoint to $j_i^!$
is well-defined. Therefore, using the Cousin resolution, we obtain that 
in order to show that a given object $\CF\in \Dmod_\lambda(\CY)$ is compact, it is sufficient
to show that
$$(j_i)^!(\CF)\in \Dmod_\lambda(\CY_i)$$
is compact for every $i$.

\medskip

By adjunction, for any $\CF_i\in \Dmod_\lambda(\CY_i)$, 
we have
$$\CHom_{\Dmod_\lambda(\CY)}((j_i)_!(\CF_i),\CF)\simeq \CHom_{\Dmod_\lambda(\CY_i)}(\CF_i,j_i^!(\CF)).$$

\medskip

Taking $\CF_i$ to be coherent, we thus reduce the assertion of the lemma to the case when $\CY$ is of the form $\on{pt}/H$, which
we will now assume. 

\sssec{Proof in the quotient case}

Consider the category $\Dmod_{\lambda}(\on{pt}/H)$, $\lambda\in \fh^*$. Note that it is zero unless $\lambda$ integrates to a character of $H$,
and in the latter case it is equivalent to the untwisted category $\Dmod(\on{pt}/H)$. 

\medskip

Thus, we can consider the case of the trivial twisting. We claim that in order to test compactness, it 
is sufficient to take $\CF'$ to be just one object, namely, 
$k_{\on{pt}/H}$. 

\medskip

Indeed, let $\pi$ denote the projection $\on{pt}\to \on{pt}/H$. 
The category $\Dmod(\on{pt}/H)$ is compactly generated by the object $\pi_!(k)$, which is a finite successive extension of
shifted copies of $k_{\on{pt}/H}$.

\medskip

Hence, if $\CHom_{\Dmod(\on{pt}/H)}(k_{\on{pt}/H},\CF)$ is compact, then so is $\CHom_{\Dmod(\on{pt}/H)}(\pi_!(k),\CF)$,
and the latter means that $\CF$ is coherent, and in particular bounded. 

\medskip

Now, according to \cite[Proposition 10.4.7]{DrGa1}, 
a bounded object of $\Dmod(\on{pt}/H)$ is safe if and only if $\CHom_{\Dmod(\on{pt}/H)}(k_{\on{pt}/H},\CF)$ is concentrated
in finitely many cohomological degrees.

\medskip

Thus, we obtain that $\CF$ is coherent and safe, and hence compact. 

\qed[\lemref{l:when compact}]

\ssec{A variant: monodromic situation}

\sssec{}

We will now consider a certain variant of \thmref{t:Y miraculous}. Let
$\pi:\wt\CY\to \CY$ be a torsor with respect to a torus $T$, and let $\lambda$ be a character of $\ft^*$.
(Note that such a datum defines a twisting on $\CY$.)

\medskip

We will now consider the full subcategory
$$\Dmod(\wt\CY)^{\lambda\on{-mon}}\subset \Dmod(\wt\CY),$$
consisting of $\lambda$-monodromic objects.

\medskip

Here is one of the possible definitions. Consider first the category
$$\Dmod(\wt\CY)^{T\on{-weak}}$$
of \emph{weakly} $T$-equivariant D-modules. This category admits a homomorphism from
$\Sym(\ft)$ into its center (called ``obstruction to equivariance"). Hence, we can view 
$\Dmod(\wt\CY)^{T\on{-weak}}$ as acted on by the monoidal category 
$\QCoh(\ft^*)$.

\medskip

We set
$$\Dmod(\wt\CY)^{\lambda\on{-mon}}:=\Dmod(\wt\CY)^{T\on{-weak}}\underset{\QCoh(\ft^*)}\otimes \QCoh(\ft^*)_{\{\lambda\}}.$$

\sssec{} 

The forgetful functor 
\begin{equation} \label{e:support lambda}
(\wh{i}_\lambda)_*:\Dmod(\wt\CY)^{\lambda\on{-mon}}\to \Dmod(\wt\CY)^{T\on{-weak}}
\end{equation}
admits a continuous right adjoint $(\wh{i}_\lambda)^!$, obtained using the 
$$\Dmod(\wt\CY)^{T\on{-weak}}\underset{\QCoh(\ft^*)}\otimes-$$
base change from the corresponding adjunction
$$(\wh{i}_\lambda)_*: \QCoh(\ft^*)_{\{\lambda\}}\rightleftarrows  \QCoh(\ft^*):(\wh{i}_\lambda)^!.$$

Since the unit of the adjunction
$$\on{Id}\to (\wh{i}_\lambda)^!\circ (\wh{i}_\lambda)_*$$
is an isomorphism for $\QCoh(\ft^*)_{\{\lambda\}}$, it is also an isomorphism for 
$\Dmod(\wt\CY)^{\lambda\on{-mon}}$. In particular, the functor \eqref{e:support lambda}
is fully faithful.

\begin{lem}
The functor
$$\Dmod(\wt\CY)^{\lambda\on{-mon}} \overset{(\wh{i}_\lambda)_*}
\hookrightarrow \Dmod(\wt\CY)^{T\on{-weak}} \overset{\oblv_{T\on{-weak}}}\longrightarrow
\Dmod(\wt\CY)$$
is fully faithful.
\end{lem}

\begin{proof} 

For $\CF_1,\CF_2\in \Dmod(\wt\CY)^{T\on{-weak}}$, we have
$$\CHom_{\Dmod(\wt\CY)}(\oblv_{T\on{-weak}}(\CF_1),\oblv_{T\on{-weak}}(\CF_2))\simeq
\CHom_{\Dmod(\wt\CY)^{T\on{-weak}}}(\CF_1,R_T\otimes \CF_2),$$
where $R_T$ is the regular representation of $T$. 

\medskip

Suppose now that $\CF_1\in \Dmod(\wt\CY)^{\lambda\on{-mon}}$. 
Then
$$\CHom_{\Dmod(\wt\CY)^{T\on{-weak}}}(\CF_1,R_T\otimes \CF_2)
\simeq 
\CHom_{\Dmod(\wt\CY)^{\lambda\on{-mon}}}(\CF_1,\underset{\mu}\oplus\,
(\wh{i}_\lambda)^!(k^\mu \otimes \CF_2)),$$
where $\mu$ runs through the set of characters of $T$, and $k^\mu$ denotes the corresponding object of $\Rep(T)$. 

\medskip

Now, the assertion of the lemma follows from the fact that if $\CF_2\in \Dmod(\wt\CY)^{\lambda\on{-mon}}$, all the terms
$$(\wh{i}_\lambda)^!(k^\mu \otimes \CF_2)$$
with $\mu\neq 0$ vanish; indeed each $k^\mu \otimes \CF_2$ belongs to $\Dmod(\wt\CY)^{(\lambda+\mu)\on{-mon}}$. 

\end{proof}

\sssec{}  \label{sss:inclusion of exact}

Note that the category $\Dmod_\lambda(\CY)$ is recovered as
$$\Dmod_\lambda(\CY)\simeq \Dmod(\wt\CY)^{T\on{-weak}}\underset{\QCoh(\ft^*)}\otimes \Vect\simeq 
\Dmod(\wt\CY)^{\lambda\on{-mon}} \underset{\QCoh(\ft^*)}\otimes \Vect,$$
where $\QCoh(\ft^*)\to \Vect$ is given by taking the fiber at $\lambda$.

\medskip

We have the tautological forgetful functor
$$(i_\lambda)_*:\Dmod_\lambda(\CY)\to \Dmod(\wt\CY)^{\lambda\on{-mon}},$$
which admits a left adjoint $(i_\lambda)^*$ and a continuous right adjoint $(i_\lambda)^!$.
These functors are obtained by base-changing the corresponding
functors for
$$\Vect \overset{(i_\lambda)_*}\longrightarrow \QCoh(\ft^*)_{\{\lambda\}}.$$

\medskip 

We have
\begin{equation} \label{e:! and *}
(i_\lambda)^*\simeq (-\otimes \ell_\lambda)\circ (i_\lambda)^!,
\end{equation} 
where $\ell_\lambda:=\Lambda^{\dim(\ft)}(\ft)[\dim(\ft)]$. 



\medskip

The functor $(i_\lambda)^!$ is conservative; hence the essential image of $(i_\lambda)_*$ generates 
$\Dmod(\wt\CY)^{\lambda\on{-mon}}$.

\sssec{}

With the above preparations, we claim:

\begin{thm} \label{t:mon miraculous}
Let $\CY$ be as in \thmref{t:Y miraculous}. Then the category
$\Dmod(\wt\CY)^{\lambda\on{-mon}}$ is proper, Serre and Gorenstein. Moreover, we have
\begin{equation} \label{e:Serre mon}
\Se_{\Dmod(\wt\CY)^{\lambda\on{-mon}}}\circ (i_\lambda)_*\simeq
(-\otimes \ell_\lambda)\circ (i_\lambda)_*\circ \Se_{\Dmod_\lambda(\CY)}.
\end{equation} 
\end{thm}

\begin{proof}

We will verify the conditions of \corref{c:main}. Since the essential image of $(i_\lambda)_*$ generates 
$\Dmod(\wt\CY)^{\lambda\on{-mon}}$, we can consider compact objects of the form
$$(i_\lambda)_*(\CG),\quad \CG\in \Dmod_\lambda(\CY)^c.$$

\medskip

For compact $\CF \in (\Dmod(\wt\CY)^{\lambda\on{-mon}})^c$ and  $\CG\in \Dmod_\lambda(\CY)^c$, we have
$$\CHom_{\Dmod(\wt\CY)^{\lambda\on{-mon}}}(\CF,(i_\lambda)_*(\CG))\simeq 
\CHom_{\Dmod_\lambda(\CY)}( (i_{\lambda})^* \CF , \CG),$$
and the latter is compact since $(i_{\lambda})^*$ is proper and $\Dmod_\lambda(\CY)$ is proper.

Hence, the properness of $\Dmod_\lambda(\CY)$ implies that of $\Dmod(\wt\CY)^{\lambda\on{-mon}}$.

\medskip

Let us show that $\Se_{\Dmod(\wt\CY)^{\lambda\on{-mon}}}$ and 
$\Se_{(\Dmod(\wt\CY)^{\lambda\on{-mon}})^\vee}\simeq \Se_{\Dmod(\wt\CY)^{-\lambda\on{-mon}}}$
are proper. By symmetry, it suffices to consider the former case. 

\medskip

It is enough to show that
for $\CF\in \Dmod_\lambda(\CY)^c$, the object
$$\Se_{\Dmod(\wt\CY)^{\lambda\on{-mon}}}\circ (i_\lambda)_*\in \Dmod(\wt\CY)^{\lambda\on{-mon}}$$
is compact, and for that it is sufficient to verify \eqref{e:Serre mon}. However, the latter follows from \eqref{e:! and *}
and \lemref{l:double right}(a) below. 

\end{proof}

\ssec{An interlude on $\fg$-modules}

In this subsection we supply some background on the self-duality of the category of
$\fg$-modules, where $\fg$ is a Lie algebra. This material will be needed for the proofs
of our main results.

\sssec{}  \label{sss:chi hat}

In this subsection we let $G$ be a reductive group and $\chi$ a character of $Z(\fg)$.  We will consider two categories associated with $\chi$. 
One is
$$\fg\mod_\chi\simeq \fg\mod\underset{Z(\fg)\mod}\otimes \Vect,$$
where $Z(\fg)\mod\to \Vect$ is given by $\underset{Z(\fg)}\otimes k_\chi$, where $k_\chi\in Z(\fg)\mod$
is the sky-scraper at $\chi$.

\medskip

The other is
$$\fg\mod_{\{\chi\}}:=\fg\mod\underset{Z(\fg)\mod}\otimes Z(\fg)\mod_{\{\chi\}},$$
where $Z(\fg)\mod_{\{\chi\}}\subset Z(\fg)\mod$ is the full subcategory of objects with set-theoretic support at $\chi\in \Spec(Z(\fg))$.

\medskip

As in \secref{sss:inclusion of exact}, we have the obvious forgetful functor 
$$(i_\chi)_*:\fg\mod_\chi\to \fg\mod_{\{\chi\}},$$
which admits a left and a continuous right adjoints, denoted $(i_\chi)^*$ and $(i_\chi)^!$, respectively.

\medskip

We have
\begin{equation} \label{e:! and * chi}
(i_\chi)^*\simeq (-\otimes \ell_\chi)\circ (i_\chi)^!,
\end{equation} 
where $\ell_\chi:=\Lambda^{\dim(\ft)}(T^*_\chi(Z(\fg)))[\dim(\ft)]$. 



\medskip

The functor $(i_\chi)^!$ is conservative; hence the essential image of $(i_\chi)_*$ generates $\fg\mod_{\{\chi\}}$.

\sssec{}  \label{sss:fg vs comp}

The categories $\fg\mod_\chi$ and $\fg\mod_{\{\chi\}}$ both carry a t-structure.

\medskip

Since the algebra $U(\fg)$ has a finite cohomological dimension, so does the full subcategory $\fg\mod_{\{\chi\}}$. 
Hence, the t-structure on $\fg\mod_{\{\chi\}}$ gives rise to one $\fg\mod_{\{\chi\}}^c$. 

\medskip

Moreover, we have
$$\fg\mod_{\{\chi\}}^c=\fg\mod_{\{\chi\}}^{\on{f.g.}},$$
where the RHS is the full subcategory of $\fg\mod_{\{\chi\}}$ that consists of objects that have non-vanishing
cohomologies in finitely many degrees and all such cohomologies being finitely generated as $U(\fg)$-modules. 

\medskip

We note, however, that if $\chi$ is \emph{irregular}, then the algebra $U(\fg)_\chi$ has an \emph{inifinite}
cohomological dimension. In particular, the t-structure on $\fg\mod_\chi$ \emph{does not} restrict to a 
t-structure on $\fg\mod_\chi^c$. We still have the inclusion
$$\fg\mod_\chi^c\subset \fg\mod_\chi^{\on{f.g.}},$$
but it is no longer an equality.

\begin{rem}
The latter circumstance can be a source of (unstabstantial, but yet annoying) difficulties. For this reason, we sometimes first prove results for $\fg\mod_{\{\chi\}}$, and then bootstrap them for $\fg\mod_\chi$.
\end{rem}

\sssec{}

The anti-involution $\xi\mapsto -\xi$ of $U(\fg)$ induces an involution on $Z(\fg)$; we denote it by $\chi\mapsto -\chi$;
in particular, the algebra $U(\fg)_{-\chi}$ canonically identifies with $(U(\fg)_\chi)^{\on{op}}$. 

\medskip

We have a canonical identification
\begin{equation} \label{e:duality g abs}
\fg\mod^\vee\simeq \fg\mod,
\end{equation}
where the evaluation functor
\begin{equation} \label{e:eval abs}
\langle - , - \rangle_{\fg} : \fg\mod\otimes \fg\mod\to\Vect
\end{equation}
is given by $\CM_1,\CM_2\mapsto \CM_2\underset{U(\fg)}\otimes \CM_1$. 

\medskip

The identification \eqref{e:duality g abs} induces an identification 
\begin{equation} \label{e:duality g gen}
\fg\mod_{\{\chi\}}^\vee\simeq \fg\mod_{\{-\chi\}},
\end{equation}
where the evaluation functor 
$$\fg\mod_{\{\chi\}}\otimes \fg\mod_{\{-\chi\}}\to \Vect$$
is obtained by precomposing \eqref{e:eval abs} with the tautological embeddings.

\medskip

Let $\BD_\fg^{\on{can}}$ denote the corresponding contravariant equivalences
$$\fg\mod^c\simeq \fg\mod^c \text{ and } \fg\mod_{\{\chi\}}^c\simeq \fg\mod_{\{-\chi\}}^c.$$

\medskip

In addition, we have a canonical identification
\begin{equation} \label{e:duality g chi}
(\fg\mod_\chi)^\vee\simeq \fg\mod_{-\chi},
\end{equation}
where the evaluation functor
$$\langle - , - \rangle_{\fg , \chi} : \fg\mod_{\chi}\otimes \fg\mod_{-\chi}\to\Vect$$
is given by $\CM_1,\CM_2\mapsto \CM_2\underset{U(\fg)_\chi}\otimes \CM_1$. 

\medskip

Let $\BD_{\fg,\chi}^{\on{can}}$ denote the corresponding contravariant equivalence
$$\fg\mod_\chi^c\simeq \fg\mod_{-\chi}^c.$$

\medskip

Note that we have a commutative diagram
$$
\CD
(\fg\mod_\chi)^\vee   @>{\sim}>>  \fg\mod_{-\chi}  \\
@V{((i_\chi)^*)^\vee}VV   @VV{(i_{-\chi})_*}V   \\
\fg\mod_{\{\chi\}}^\vee  @>>>  \fg\mod_{\{-\chi\}}.
\endCD
$$

In other words, we have a canonical isomorphism of contravariant functors 
$$(-\otimes \ell_\chi) \circ \BD_\fg^{\on{can}}\circ (i_\chi)_* \simeq (i_{-\chi})_*\circ \BD_{\fg,\chi}^{\on{can}}, \quad 
\fg\mod_\chi^c\rightrightarrows \fg\mod_{\{-\chi\}}^c.$$

\sssec{}

Let $H\subset G$ be any subgroup. Recall that if $\bC$ is a dualizable category acted on by $H$, then we have a canonical 
identification
\begin{equation} \label{e:dual equiv}
(\bC^\vee)^H\simeq (\bC^H)^{\vee}.
\end{equation}

\medskip

The corresponding pairing 
$$\langle-,-\rangle_{\bC,H}:(\bC^\vee)^H\otimes \bC^H\to \Vect$$
is the composition
$$(\bC^\vee)^H\otimes \bC^H\to (\bC^\vee\otimes \bC)^H\to \Vect^H\simeq \Dmod(\on{pt}/H)\to \Vect,$$
where the second arrow is induced by the pairing 
$$\langle-,-\rangle:\bC^\vee\otimes \bC\to \Vect,$$
and the third arrow is the functor of 
\emph{renormalized} de Rham cohomology (see \cite[Sect. 9.1]{DrGa1}), i.e., 
the \emph{renormalized} direct image functor (see \cite[Sect. 9.3]{DrGa1}) along $\on{pt}/H\to \on{pt}$. 

\medskip

With respect to the identification \eqref{e:dual equiv}, the functor dual to
$$\oblv_H:\bC^H\to \bC$$
is the functor
$$\Av^H_*:\bC^\vee\to (\bC^\vee)^H,$$
and vice versa. 

\sssec{}

Thus, the identifications \eqref{e:duality g abs}, \eqref{e:duality g gen} and \eqref{e:duality g chi}
induce the identifications
\begin{equation} \label{e:duality g H}
(\fg\mod^H)^\vee\simeq \fg\mod^H,
\end{equation}
\begin{equation} \label{e:duality g H gen}
(\fg\mod^H_{\{\chi\}})^\vee\simeq \fg\mod^H_{-\{\chi\}},
\end{equation}
\begin{equation} \label{e:duality g H chi}
(\fg\mod^H_\chi)^\vee\simeq \fg\mod^H_{-\chi}.
\end{equation}

We will denote the corresponding pairings as follows:

$$ \langle -, -\rangle_{\fg,H} : \fg\mod^H \otimes \fg\mod^H \to \Vect,$$
$$ \langle -, -\rangle_{\fg,H} : \fg\mod^H_{\{ \chi\}} \otimes \fg\mod^H_{\{ -\chi\}} \to \Vect,$$
$$ \langle -, -\rangle_{\fg,\chi,H} : \fg\mod^H_{\chi} \otimes \fg\mod^H_{-\chi} \to \Vect.$$

We keep the same notations for the corresponding contravariant equivalences
$$\BD_\fg^{\on{can}}:(\fg\mod^H)^c\simeq (\fg\mod^H)^c \text{ and }
\BD_\fg^{\on{can}}:(\fg\mod_{\{\chi\}}^H)^c\simeq (\fg\mod_{-\{\chi\}}^H)^c,$$
and 
$$\BD_{\fg,\chi}^{\on{can}}:(\fg\mod_\chi^H)^c\simeq (\fg\mod_{-\chi}^H)^c.$$

We have:
$$\oblv_H\circ \BD_\fg^{\on{can}}\simeq \BD_\fg^{\on{can}}\circ \oblv_H \text{ and }
\oblv_H\circ \BD_{\fg,\chi}^{\on{can}}\simeq \BD_{\fg,\chi}^{\on{can}}\circ \oblv_H.$$

\medskip

The functors $(i_\chi)^*,(i_\chi)_*,(i_\chi)^!$ induce functors between the corresponding equivariant
categories, and the latter are compatible with the corresponding functors $\oblv_H$ and $\Av^H_*$.

\ssec{A reminder on Localization Theory}

\sssec{}   \label{sss:loc}

Let $\lambda$ be a character of $\ft$ that corresponds to $\chi$ under the Harish-Chandra map. To $\lambda$ 
we assign a TDO $\on{D}_\lambda$ on the flag variety $X$ of $G$. 

\medskip

\noindent NB: Unlike \cite{Kim}, we do \emph{not} apply the $\rho$-shift when we assign $\on{D}_\lambda$ to $\lambda$.
In particular, $\lambda=0$ corresponds to the untwisted $\on{D}$.

\medskip

Consider the functor
\begin{equation} \label{e:sections}
\Gamma:\Dmod_\lambda(X)\to \fg\mod_\chi.
\end{equation}

By \cite{BB}, the functor $\Gamma$ admits a fully faithful left adjoint, denoted $\on{Loc}$;
both these functors are compatible with the action of $G$.  

\medskip

The functors $\Loc$ and $\Gamma$ define functors between the categories
$$ \fg\mod_\chi^H \text{ and } \Dmod_\lambda(H\backslash X)\simeq \Dmod_\lambda(X)^H$$
with the same adjunction properties; these functors are compatible with the induction and
forgetful functors $\oblv_H$ and $\Av^H_*$. 

\sssec{}   \label{sss:loc and dual}

The identifications
$$\Dmod_\lambda(X)^\vee\simeq \Dmod_{-\lambda}(X) \text{ and } (\fg\mod_\chi)^\vee\simeq \fg\mod_{-\chi}$$
are compatible as follows. First, we note that if $\lambda$ corresponds to $\chi$, then $-\lambda-2\rho$
corresponds to $-\chi$. Now, the functor 
$$\Gamma^\vee: (\fg\mod_\chi)^\vee\to \Dmod_\lambda(X)^\vee,$$
dual to \eqref{e:sections}, identifies canonically with 
$$\fg\mod_{-\chi} \overset{\Loc}\longrightarrow \Dmod_{-\lambda-2\rho}(X) \overset{-\otimes \omega^{-1}_X}\longrightarrow 
\Dmod_{-\lambda}(X),$$
where $\omega_X$ is the dualizing complex on $X$, and we use the fact that $\omega_X\simeq \CO(-2\rho)[\dim(X)]$. 

\medskip

In other words, we have an isomorphism of contravariant functors 
$$\Loc \circ  \BD_{\fg,\chi}^{\on{can}}\simeq 
(-\otimes \CO(-2\rho))[\dim(X)]\circ \BD_X^{\on{Verdier}}\circ \Loc, \quad 
(\fg\mod_\chi)^c\rightrightarrows \Dmod_{-\lambda-2\rho}(X)^c.$$
 
\medskip
 
Similar identifications pass on to the corresponding $H$-equivariant categories. 

\sssec{}

We will now consider the following variant of the adjunction $(\Loc,\Gamma)$. Namely, we consider the
base affine space
$$\wt{X}\to X,$$
and for a given $\lambda$ we consider the corresponding category 
$$\Dmod(\wt{X})^{\lambda\on{-mon}}.$$

Taking global sections on $\wt{X}$, and then taking $T$-invariants, we obtain a functor
$$\Gamma':\Dmod(\wt{X})^{\lambda\on{-mon}}\to \fg\mod_{\{\chi\}}.$$

\medskip

We have a commutative diagram
$$
\CD
\Dmod_{\lambda}(X) @>{(i_\lambda)_*}>>  \Dmod(\wt{X})^{\lambda\on{-mon}} \\
@V{\Gamma}VV   @VV{\Gamma'}V  \\
\fg\mod_\chi @>{(i_\chi)_*}>>  \fg\mod_{\{\chi\}},
\endCD
$$

\medskip

The main advantage of the functor $\Gamma'$ (unlike that of $\Gamma$ of
\eqref{e:sections}) is given by the following lemma:

\begin{lem} \label{l:global good}
The functor $\Gamma'$ sends compact objects to compact ones.
\end{lem}

\begin{proof}
Follows from the fact that the functor $\Gamma$ of \eqref{e:sections} sends 
$\Dmod_{\lambda}(X)^c$ to $\fg\mod_\chi^{\on{f.g.}}$.
\end{proof} 

\sssec{}

The functor $\Gamma'$ admits a left adjoint, denoted $\Loc'$, but the latter is \emph{no longer} fully faithful,
see \secref{sss:fix Loc'}. 

\medskip

It follows from \lemref{l:global good} that the functor $\Gamma'$ 
admits also a continuous \emph{right} adjoint, denoted $\coLoc'$. 

\medskip

One easily shows that the functor $\coLoc'$
is also compatible with the $G$-actions. In particular, we have the functors
$\Loc',\Gamma',\coLoc'$ between the corresponding $H$-equivariant
categories, compatible with the functors $\oblv_H$ and $\Av^H_*$. 

\medskip

Since 
$$\coLoc':\fg\mod_{\{\chi\}}^H\to \Dmod(H\backslash \wt{X})^{\lambda\on{-mon}}$$
is continuous, we obtain that the functor 
$$\Gamma':\Dmod(H\backslash \wt{X})^{\lambda\on{-mon}}\to \fg\mod_{\{\chi\}}^H$$
preserves compactness. 

\sssec{} \label{sss:fix Loc'}

As was mentioned above, 
the key difference between the $(\Loc,\Gamma)$ and $(\Loc',\Gamma')$-adjunctions is that the functor $\Loc'$ is
no longer fully faithful, but its failure to be fully faithful is controllable. 

\medskip

Namely, consider the Harish-Chandra map
$Z(\fg)\to \Sym(\ft)$, and consider the algebra
$$U(\fg)^\sim:=U(\fg)\underset{Z(\fg)}\otimes \Sym(\ft).$$

\medskip

Denote $\fg\mod^\sim:=U(\fg)^\sim\mod$, and let
$$\fg\mod^\sim_{\{\lambda\}}\subset \fg\mod^\sim$$
be the full subcategory consisting of objects that are set-theoretically supported at $\lambda$ as $\Sym(\ft)$-modules.

\medskip

Then the functor 
$$\Gamma': \Dmod(\wt{X})^{\lambda\on{-mon}}\to \fg\mod_{\{\chi\}}$$
factors as
$$\Dmod(\wt{X})^{\lambda\on{-mon}} \overset{\Gamma^\sim}\longrightarrow \fg\mod^\sim_{\{\lambda\}}\to 
\fg\mod_{\{\chi\}},$$
where $\fg\mod^\sim_{\{\lambda\}}\to \fg\mod_{\{\chi\}}$ is the forgetful functor. 

\medskip

The above functor $\Gamma^\sim$ also admits a left adjoint, denoted $\Loc^\sim$, and the latter functor \emph{is}
fully faithful.

\medskip

From here we obtain:

\begin{lem} \label{l:Loc sim}
The endofunctor $\Gamma' \circ \Loc'$ of $\fg\mod_{\{\chi\}}$ is given by 
$$\CM\mapsto \CM\underset{Z(\fg)}\otimes \CQ,$$
where $\CQ$ is $\Sym(\ft)$,regarded as an $Z(\fg)$-module
\end{lem}

Note that the above module $\CQ$ is (locally) free. 

\begin{cor} \label{c:Loc sim}
The endofunctor $\Gamma' \circ \coLoc'$ of $\fg\mod_{\{\chi\}}$ is given by 
$$\CM\mapsto \CM\underset{Z(\fg)}\otimes \CQ^\vee,$$
where $\CQ^\vee$ is the dual of $\CQ$.
\end{cor} 

\sssec{}

The entire discussion in \secref{sss:fix Loc'}, and the conclusions of \lemref{l:Loc sim} and \corref{c:Loc sim}
transfer to the $H$-equivariant situation for any given $H\subset G$.

\ssec{Examples arising from representation theory}

\sssec{}

In this subsection we will prove the following:

\begin{thm}   \label{t:gH miraculous}
Let $H\subset G$ be \emph{spherical}, i.e., $H$ 
has finitely many orbits on the flag vatiety $X$. Then: 

\smallskip

\noindent{\em(a)} 
The category $\fg\mod_\chi^H$ is proper, Serre and Gorenstein, and the functors 
$\Se_{\fg\mod_\chi^H}$ and $\PsId_{\fg\mod_\chi^H}$ are mutually inverse equivalences.

\smallskip

\noindent{\em(b)} Ditto for the category $\fg\mod_{\{\chi\}}^H$.

\end{thm} 

\begin{rem}
Note that when $\chi$ is regular, the assertion of \thmref{t:gH miraculous}(a) follows immediately from
the fact that in this case the functor $\Loc$ is an equivalence, and \thmref{t:Y miraculous}.
\end{rem}

\sssec{}

First, we have a lemma:

\begin{lem}  \label{l:double right}
Let $F:\bC\to \bD$ be proper.  

\smallskip

\noindent{\em(a)} We have: $\Se_\bD\circ F\simeq (F^R)^R\circ \Se_\bC$.

\smallskip

\noindent{\em(b)} If $F$ is fully faithful, then $\Se_\bC\simeq F^R\circ \Se_\bD\circ F$. 

\end{lem} 

\begin{proof}
For $\bc\in \bC^c$ and $\bd\in \bD^c$, we have
$$\CHom_\bD(\bd,\Se_\bD\circ F(\bc))\simeq \CHom_{\bD}(F(\bc),\bd)^*\simeq \CHom_{\bC}(\bc,F^R(\bd))^*$$
and
$$\CHom_\bD(\bd,(F^R)^R\circ \Se_\bC(\bc))\simeq \CHom_\bC(F^R(\bd),\Se_\bC(\bc))\simeq 
\CHom_\bC(\bc,F^R(\bd))^*.$$

This proves point (a). For point (b) we note that the fact that $F$ is fully faithful implies that $(F^R)^R$ is such
as well. In particular $F^R\circ (F^R)^R\simeq \on{Id}_\bC$. 
Composing the isomorphism of point (a) with $F^R$, we arrive at the assertion of point (b). 

\end{proof}

\sssec{Proof of \thmref{t:gH miraculous}, Step 1}

We will verify that the conditions of \corref{c:main} hold. 

\medskip

We first verify that the categories in question are proper. For $\fg\mod_\chi^H$, this is a formal consequence of the
fact that $\Loc$ is proper and fully faithful, and the properness of $\Dmod_{\lambda}(X)$.

\medskip

The case of $\fg\mod_{\{\chi\}}^H$ follows formally from that of $\fg\mod_\chi^H$, using the $((i_\chi)^*,(i_\chi)_*,(i_\chi)^!)$
adjunctions as in the proof of \thmref{t:mon miraculous}.

\sssec{Proof of \thmref{t:gH miraculous}, Step 2}

We will now show that the functors $\Se_{\fg\mod_{\{\chi\}}^H}$ and 
$\Se_{(\fg\mod_{\{\chi\}}^H)^\vee}\simeq \Se_{\fg\mod_{\{-\chi\}}^H}$
are proper. By symmetry, it suffices to deal with the former functor. 

\medskip 

By \lemref{l:double right}(a), we have
$$\coLoc'\circ \Se_{\fg\mod_{\{\chi\}}^H}\simeq \Se_{\Dmod(H\backslash \wt{X})^{\lambda\on{-mon}}} \circ \on{Loc}'.$$

Composing with $\Gamma'$, and using \corref{c:Loc sim}, we obtain:
$$(- \underset{Z(\fg)}\otimes \CQ^\vee)\circ \Se_{\fg\mod_{\{\chi\}}^H}\simeq \Gamma'\circ \Se_{\Dmod(H\backslash \wt{X})^{\lambda\on{-mon}}} 
\circ \on{Loc}'.$$

Note that all the functors in the RHS preserve compactness (for $\Se_{\Dmod(H\backslash \wt{X})^{\lambda\on{-mon}}}$
we are using \thmref{t:mon miraculous}). 

\medskip

Hence, it remains to show that if $\CM\in \fg\mod_{\{\chi\}}^H$ is such that $\CM\underset{Z(\fg)}\otimes \CQ^\vee$
is compact, then so is $\CM$.  However, if $\CM\underset{Z(\fg)}\otimes \CQ^\vee$ is compact, then so is
$\CM\underset{Z(\fg)}\otimes \CQ^\vee\underset{Z(\fg)}\otimes \CQ$, and $\CM$ is a direct summand of the latter.

\sssec{Proof of \thmref{t:gH miraculous}, Step 3}

We will now show that the functors $\Se_{\fg\mod_\chi^H}$ and 
$\Se_{(\fg\mod_\chi^H)^\vee}\simeq \Se_{\fg\mod_{-\chi}^H}$
are proper. Again, by symmetry, it suffices to deal with the former functor. 

\medskip

Note that the functor 
$$(i_\chi)_*:\fg\mod_\chi^H\to \fg\mod_{\{\chi\}}^H$$
is conservative, because both
$$(i_\chi)_*:\fg\mod_\chi\to \fg\mod_{\{\chi\}} \text{ and } \oblv_H:\fg\mod_\chi^H\to \fg\mod_\chi$$
are conservative. 

\medskip

Hence, $\fg\mod_\chi^H$ is compactly generated by the essential image of $(\fg\mod_{\{\chi\}}^H)^c$ 
under $(i_\chi)^*$. Hence, it is enough to show that the functor
$$\Se_{\fg\mod_\chi^H}\circ (i_\chi)^*$$
sends compacts to compacts. 

\medskip

However, from \lemref{l:double right}(a), we obtain
$$\Se_{\fg\mod_\chi^H}\circ (i_\chi)^*\simeq (i_\chi)^!\circ \Se_{\fg\mod_{\{\chi\}}^H},$$
which isomorphic to $(i_\chi)^*\circ \Se_{\fg\mod_{\{\chi\}}^H}$ up to tensoring with $\ell_\chi$. 

\medskip

Now the assertion follows from the fact that $\Se_{\fg\mod_{\{\chi\}}^H}$ preserves compactness,
proved in Step 2.

\qed[\thmref{t:gH miraculous}]

\ssec{An application: a theorem of \cite{BBM}}

\sssec{}

As an application we will now (re)prove the following result (whose case (b) is a theorem from \cite{BBM}). 

\medskip

We take $H$ to be the subgroup $N$, the unipotent radical of a Borel in $G$. Take $\bC$ to be either

\smallskip

\noindent(a) $\fg\mod_\chi$, or 

\smallskip

\noindent(b) $\Dmod_\lambda(X)$.

\medskip

Recall the intertwining functor
$$\Upsilon:=\Av^N_*\circ \oblv_{N^-}:\bC^{N^-}\to \bC^N,$$
see \cite[Sect. 1.4]{Kim}, and similarly
$$\Upsilon^-:\bC^N\to \bC^{N^-}.$$

\sssec{}

We have:

\begin{thm}  \label{t:BBM}
The category $\bC^N$ is Serre, and we have a canonical isomorphism
$$\Se_{\bC^N}\simeq \Upsilon\circ \Upsilon^-[2\dim(X)].$$
\end{thm} 

\begin{proof}

The assertion for $\bC=\Dmod_\lambda(X)$ is the combination of \thmref{t:Y miraculous} above and 
\cite[Theorem 3.4.2]{Kim}. 

\medskip

For $\bC=\fg\mod_\chi$ we have
\begin{multline*} 
\Se_{\fg\mod_\chi^N}\overset{\text{\lemref{l:double right}(b)}}\simeq \Gamma\circ \Se_{\Dmod_\lambda(N\backslash X)}\circ \on{Loc}\simeq \\
\simeq \Gamma\circ \Upsilon \circ \Upsilon^-\circ \on{Loc}[2\dim(X)]\simeq \Upsilon \circ \Gamma\circ \on{Loc}\circ \Upsilon^- [2\dim(X)]\simeq 
\Upsilon\circ \Upsilon^-[2\dim(X)],
\end{multline*}
since the functors $\Gamma$ and $\on{Loc}$ commute with all averaging functors.

\end{proof} 

\section{Serre functor and contragredient duality}  \label{s:contr}

The theme of this section is to compare the canonical duality functor $\BD^{\on{can}}_{\fg,\chi}$ on a category of the form $\fg\mod^H_\chi$
with various kinds of \emph{contragredient duality} functors. 

\ssec{A warm-up: algebraic representations}

We begin with the simplest case, namely, when $\fg=\fh$. 

\sssec{}

Let $H$ be an algebraic group. Consider the category
$$\fh\mod^H\simeq \Rep(H).$$

We note that there are two \emph{different} identifications
$$\Rep(H)^\vee \simeq \Rep(H).$$

One is given by \eqref{e:duality g H}; In this section, we will denote the corresponding contravariant self-equivalence of 
$(\fh\mod^H)^c$ by $\BD^{\on{can}}_{\fh,H}$ rather than $\BD^{\on{can}}_{\fh}$.

\medskip

The other is given by ind-extending the
contravariant self-equivalence
$$\BD^{\on{contr}}_H:\Rep(H)^c\to \Rep(H)^c, \quad V\mapsto V^\vee.$$
given by the passage to the \emph{dual} representation. 

\sssec{}

The composite of these two identifications is a self-equivalence of $\Rep(H)$. 
It is given by ind-extending the (covariant) self-equivalence
$$\BD^{\on{contr}}_H\circ \BD^{\on{can}}_{\fh,H}:\Rep(H)^c\to \Rep(H)^c.$$

\medskip

We claim:

\begin{prop} \label{p:fin dim dual}
The functor $\BD^{\on{contr}}_H\circ \BD^{\on{can}}_{\fh,H}$ is given by tensoring with the line $\ell_H:=\Lambda^{\dim(H)}(\fh)[\dim(\fh)]$. 
\end{prop}

\begin{proof}

We first establish the corresponding isomorphism after composing with 
\begin{equation} \label{e:alg grp}
\oblv_H:\Rep(H)\simeq \fh\mod^H\overset{\oblv_H}\longrightarrow \fh\mod.
\end{equation}

For the latter, we have to establish an isomorphism
$$\BD^{\on{can}}_\fh\circ \oblv_H(V)\simeq \oblv_H(V^\vee)\otimes \ell_H, \quad V\in \Rep(H)^c.$$

Taking $\CHom_{\fh\mod}$ of both sides into $\CM\in \fh\mod$, we obtain that we need to establish a functorial isomorphism
between
$$\CHom_{\fh\mod}(\BD^{\on{can}}_\fh\circ \oblv_H(V),\CM)\simeq \on{C}_\bullet(\fh,\oblv_H(V)\otimes \CM)$$
and
$$\CHom_{\fh\mod}(\oblv_H(V^\vee)\otimes \ell_H,\CM)\simeq \ell_H^{-1}\otimes \on{C}^\bullet(\fh,\oblv_H(V)\otimes \CM).$$

\medskip

The required isomorphism follows now from 
$$\on{C}_\bullet(\fh,-) \simeq \ell_H^{-1}\otimes \on{C}^\bullet(\fh,-).$$

\medskip

Thus, we obtain that the endo-functors
$$(-\otimes \ell^{-1}_H)\circ \BD^{\on{contr}}_H\circ \BD^{\on{can}}_{\fh,H}:\Rep(H)\to \Rep(H) \text{ and } \on{Id}_{\Rep(H)}$$
are both given by objects in 
$$\Rep(H)^\vee\otimes \Rep(H)\simeq \Rep(H)\otimes \Rep(H)$$ that become canonically isomorphic after applying
the functor 
\begin{equation} \label{e:alg grp heart}
\on{Id}_{\Rep(H)}\otimes \oblv_H:\Rep(H)\otimes \Rep(H)\to \Rep(H)\otimes \fh\mod.
\end{equation} 

Since the latter is t-exact and conservative, we obtain that the above two objects both lie in 
$(\Rep(H)\otimes \Rep(H))^\heartsuit$. 

\medskip

Now, the restriction of the functor \eqref{e:alg grp heart} to $(\Rep(H)\otimes \Rep(H))^\heartsuit$
is \emph{fully faithful}. Hence, the above two objects are isomorphic in $\Rep(H)\otimes \Rep(H)$
itself.

\end{proof}

\sssec{}   \label{sss:from top}

In what follows we will need the following comparison result. Let $\bC$ be a DG category acted on by 
a reductive group $H$. 

\medskip

We claim that there is a natural transformation
\begin{equation} \label{e:top dr}
\Av^H_* \circ \oblv_H \to (-\otimes \ell_H^{-1})\otimes .
\end{equation} 

Indeed, the functor $\Av^H_* \circ \oblv_H$ is given by tensoring with $\on{C}^\bullet_{\dr}(H)$,  viewed as 
an object of $\Dmod(\on{pt}/H)$, i.e., the direct image of $k\in \Vect\simeq \Dmod(\on{pt})$ along the
projection $\on{pt}\to \on{pt}/H$. 

\medskip

Now, for any algebraic group, we have a canonical isomorphism 
$$\on{C}^\bullet_{\dr}(H)\simeq \on{C}^\bullet(\fh,R_H),$$
and if $H$ is reductive, the map
$$\on{C}^\bullet(\fh,k)\to \on{C}^\bullet(\fh,R_H)$$
is an isomorphism. Finally, we identify $\ell_H^{-1}$ with the top cohomology of $\on{C}^\bullet(\fh,k)$, using
$$\on{C}^\bullet(\fh,k)\simeq \on{C}_\bullet(\fh,k)\otimes \ell_H^{-1}.$$

\medskip

This provides the desired map
$$\on{C}^\bullet_{\dr}(H)\to \ell_H^{-1}\otimes k_{\on{pt}/H}$$
in $\Dmod(\on{pt}/H)$.

\ssec{The case of category $\CO$}  \label{ss:O}

In this subsection we study (the derived version of) the usual category $\CO$, i.e., the category  $\fg\mod_\chi^N$, where $N$ is the unipotent radical of a Borel subgroup.

\medskip

We will see that the discrepancy between the canonical duality functor $\BD^{\on{can}}_{\fg,\chi}$
and the \emph{usual contragredient duality} for category $\CO$ is given by the long intertwining functor $\Upsilon$. 

\sssec{}

Consider the equivalence 
\begin{equation}   \label{e:to be contr}
(\fg\mod_\chi^N)^\vee \simeq \fg\mod_{-\chi}^{N^-}
\end{equation} 
equal to the composition
$$(\fg\mod_\chi^N)^\vee \simeq \fg\mod_{-\chi}^N \overset{\Se_{\fg\mod_{-\chi}^N}}\longrightarrow 
\fg\mod_{-\chi}^N  \overset{\Upsilon^{-1}}\longrightarrow \fg\mod_{-\chi}^{N^-},$$
where the first arrow is the equivalence \eqref{e:duality g H chi}.

\medskip

Note that from \thmref{t:BBM} we obtain: 

\begin{cor}  \label{c:prel AG}
The identification \eqref{e:to be contr} is canonically isomorphic to
$$(\fg\mod_\chi^N)^\vee \simeq \fg\mod_{-\chi}^N \overset{\Upsilon^-[2\dim(X)]}\longrightarrow \fg\mod_{-\chi}^{N^-},$$
where the first arrow is the identification of \eqref{e:duality g H chi}.
\end{cor}

\medskip

Recall that $\BD_{\fg,\chi}^{\on{can}}$ denotes the contravariant equivalence
$$(\fg\mod_\chi^N)^c\simeq (\fg\mod_{-\chi}^N)^c,$$
corresponding to \eqref{e:duality g H chi} (note that, according to \secref{sss:loc and dual}, the functor $\Loc$ intertwines
$\BD_{\fg,\chi}^{\on{can}}$ with Verdier duality on $N\backslash X$). 

\medskip

Let $\BD_{\fg,\chi}^{\on{contr}}$ denote the contravariant equivalence
\begin{equation}   \label{e:contr O}
(\fg\mod_\chi^N)^c\simeq (\fg\mod_{-\chi}^{N^-})^c,
\end{equation} 
corresponding to \eqref{e:to be contr} (and similarly with the roles of $\chi/-\chi$ or $N/N^-$ swapped). 

\medskip

We can rephrase \corref{c:prel AG} as follows: 

\begin{cor} \label{c:AG}
The (covariant) equivalence 
$$\BD_{\fg,-\chi}^{\on{contr}} \circ \BD_{\fg,\chi}^{\on{can}}: (\fg\mod_\chi^N)^c \to (\fg\mod_\chi^{N^-})^c$$
is given by the functor $\Upsilon^-[2\dim(X)]$.
\end{cor}

We will now show that the functor $\BD_{\fg,\chi}^{\on{contr}}$ is (the derived version of) the usual
contragredient duality on category $\CO$. 

\medskip

In particular, we obtain that \corref{c:AG} reproduces
the result of \cite[Theorem 1.4.6]{AG} that describes the interaction of the contragredient
and canonical dualities on category $\CO$.

\sssec{}

Let us denote by \begin{equation} \label{e:contr dualization}
\CM\mapsto \CM^\vee, \quad (\fg\mod_\chi^N)^{\heartsuit} \to ((\fg\mod_{-\chi}^{N^-})^\heartsuit)^{\on{op}}
\end{equation}
the contravariant functor given by assigning to $\CM$ the subspace of the abstract dual $\CM^*$, consisting of $N^-$-finite vectors.

\sssec{}

On the subcategory $$ (\fg\mod_\chi^N)^{\heartsuit,\on{f.g.}} \subset (\fg\mod_\chi^N)^{\heartsuit} $$ consisting of $\CM$ for which 
$\oblv_N (\CM) \in (\fg\mod_{\chi})^{\heartsuit}$ is finitely generated, the functor $\CM \mapsto \CM^{\vee}$ admits the following description: 

\medskip

Write $$\CM\simeq \underset{\mu}\oplus\, \CM_{(\mu)},$$ a direct sum of generalized eigenspaces with respect to $\ft$,
which are known to be finite-dimensional. Then $$\CM^\vee\simeq \underset{\mu}\oplus\, (\CM_{(\mu)})^*$$ (which acquires a natural $\fg$-module structure). In other words, $\CM \mapsto \CM^{\vee}$ is the "usual" contragradient duality.

\medskip

Moreover, it is known that in this way we obtain a contravariant equivalence 
$$(\fg\mod_\chi^N)^{\heartsuit,\on{f.g.}} \simeq (\fg\mod_{-\chi}^{N^-})^{\heartsuit,\on{f.g.}}.$$

\sssec{}

We claim:

\begin{thm}  \label{t:descr contr}
The ind-extension of the contravariant equivalence $\BD_{\fg,\chi}^{\on{contr}}$ of \eqref{e:contr O}
\begin{equation} \label{e:contr ext}
\BD_{\fg,\chi}^{\on{contr}}:\fg\mod_\chi^N\to (\fg\mod_{-\chi}^{N^-})^{\on{op}}
\end{equation} 
is t-exact when restricted to $ (\fg\mod_{\chi}^N)^{\on{f.g.}}$. The corresponding contravariant functor 
$$(\fg\mod_\chi^N)^{\heartsuit, \on{f.g.}} \to ((\fg\mod_{-\chi}^{N^-})^\heartsuit)^{\on{op}}$$
is given by the functor $\CM\mapsto \CM^\vee$ of \eqref{e:contr dualization}. 
\end{thm}

\sssec{Proof of \thmref{t:descr contr}, Step 1}

It is enough to show that the functor \eqref{e:contr ext} sends an object $\CM\in (\fg\mod_\chi^N)^{\heartsuit,\on{f.g.}}$
to $\CM^\vee$.  

\medskip

By definition, for $\CM'\in \fg\mod_{-\chi}^{N^-}$, we have
\begin{equation} \label{e:Hom into contr}
\CHom_{\fg\mod_{-\chi}^{N^-}}(\CM',\BD_{\fg,\chi}^{\on{contr}}(\CM))\simeq
\langle \CM, \Upsilon(\CM') \rangle_{\fg,\chi,N}^*\simeq
\langle \CM,\CM'\rangle_{\fg,\chi}^*.
\end{equation} 

\medskip

Note that if $\CM'\in (\fg\mod_{-\chi}^{N^-})^{\leq 0}$, then 
$$\langle \CM,\CM'\rangle_{\fg,\chi}=\CM'\underset{U(\fg)_\chi}\otimes \CM$$ 
belongs to $\Vect^{\leq 0}$. This readily implies that $\BD_{\fg,\chi}^{\on{contr}}(\CM)\in (\fg\mod_{-\chi}^{N^-})^{\geq 0}$.

\medskip

Further, it is easy to see that for any $\CM\in (\fg\mod_\chi^N)^\heartsuit$, we have 
$$H^0\left(\langle \CM,\CM'\rangle^*\right)\simeq \Hom(H^0(\CM'),\CM^\vee),$$
hence $H^0(\BD_{\fg,\chi}^{\on{contr}}(\CM))\simeq \CM^\vee$. 

\medskip

Thus, it remains to show that $\BD_{\fg,\chi}^{\on{contr}}(\CM)\in (\fg\mod_{-\chi}^{N^-})^\heartsuit$. 

\sssec{Proof of \thmref{t:descr contr}, Step 2}

Every object of $(\fg\mod_\chi^N)^{\heartsuit,\on{f.g.}}$ admits a surjection from a finite direct sum of objects of the form
$$M_i:=U(\fg)_\chi \underset{U(\fn)}\otimes U(\fn)_i,$$
where $U(\fn)_i=U(\fn)/U(\fn)\cdot \fn^i$. 

\medskip

Hence, using the exactness of $\CM \mapsto H^0 (\BD^{\on{contr}}_{\fg,\chi} (\CM)) \simeq \CM^{\vee}$,  it suffices to show that 
$$M_i^\vee\simeq  H^0(\BD_{\fg,\chi}^{\on{contr}}(M_i)) \to \BD_{\fg,\chi}^{\on{contr}}(M_i)$$
is an isomorphism. 

\medskip

Furthermore, objects of the form
$$M^-_j:=U(\fg)_{-\chi} \underset{U(\fn^-)}\otimes U(\fn^-)_j$$
compactly generate $\fg\mod_{-\chi}^{N^-}$. Hence, it suffices to show that the map
$$
\CHom_{\fg\mod_{-\chi}^{N^-}}(M^-_j,M_i^\vee)\to \CHom_{\fg\mod_{-\chi}^{N^-}}(M^-_j,\BD_{\fg,\chi}^{\on{contr}}(M_i))
$$
is an isomorphism. 

\medskip

Each $M_i$ (resp., $M^-_j$) admits a filtration with subquotients of the form $M_1$ (resp., $M^-_1$). Hence, it is
enough to show that the map
\begin{equation} \label{e:comp Verma and dual}
\CHom_{\fg\mod_{-\chi}^{N^-}}(M^-_1,M_1^\vee)\to \CHom_{\fg\mod_{-\chi}^{N^-}}(M^-_1,\BD_{\fg,\chi}^{\on{contr}}(M_1))
\end{equation}
is an isomorphism. 

\medskip

We know that \eqref{e:comp Verma and dual} is an isomorphism at the level of $H^0$. Hence, it suffices to show
that both sides in  \eqref{e:comp Verma and dual} are acyclic off degree $0$. 

\sssec{Proof of \thmref{t:descr contr}, Step 3}

We have:
$$\CHom_{\fg\mod_{-\chi}^{N^-}}(M^-_1,\BD_{\fg,\chi}^{\on{contr}}(M_1))\simeq
(M^-_1\underset{U(\fg)_\chi}\otimes M_1)^*,$$
while 
$$M^-_1\underset{U(\fg)_\chi}\otimes M_1 \simeq k\underset{U(\fn^-)}\otimes U(\fg)_\chi\underset{U(\fn)}\otimes k,$$
and the above expression is indeed acyclic off degree $0$, as it is known that $U(\fg)$ is free as a module over 
$U(\fn^-)\otimes U(\fn)\otimes Z(\fg)$. 

\medskip

To prove that $\CHom_{\fg\mod_{-\chi}^{N^-}}(M^-_1,M_1^\vee)$ is acyclic off degree $0$,
we note that for any object
$\CN\in \fg\mod_{-\chi}^{N^-}$, we have 
$$\CHom_{\fg\mod_{-\chi}^{N^-}}(M^-_1,\CN) \simeq \on{C}^\bullet(\fn^-,\CN).$$

So we need to show that that $\on{C}^\bullet(\fn^-,M_1^\vee)$ is acyclic off degree $0$.

\medskip

We note that $M_1$ admits a filtration by Verma modules. Hence, it is enough to show that for a Verma module $M$, the object
$\on{C}^\bullet(\fn^-,M^\vee)$ is acyclic off degree $0$.

\medskip

Now, $M^\vee$ is a \emph{dual Verma module}, which is isomorphic to $R_{N^-}$ as a module over $\fn^-$.
This implies that $\on{C}^\bullet(\fn^-,M^\vee)$ is acyclic off degree $0$, as required.

\ssec{The case of Harish-Chandra modules}  \label{ss:gK}

We will now consider the main case of interest in this paper: the interaction of the canonical and contragredient
duality functors on (the derived version of) the category of Harish-Chandra modules. 

\medskip

We will see that the discrepancy between the two is given by the pseudo-identity functor.

\sssec{}

We now consider a symmetric subgroup $K \subset G$, i.e. $K = G^{\theta}$ for an involution $\theta$ of $G$. Such a subgroup 
is (connected) reductive, and is spherical 
(i.e. has finitely many orbits on the flag variety $X$)\footnote{
In fact, everything in this section remains valid when $K$ is any (connected) reductive spherical subgroup in $G$.}.

\medskip

Consider the equivalence
\begin{equation} \label{e:to be contr gK}
(\fg\mod_\chi^K)^\vee \simeq \fg\mod_{-\chi}^K \overset{\Se_{\fg\mod_{-\chi}^K}}\longrightarrow \fg\mod_{-\chi}^K
\overset{-\otimes \ell_K}\longrightarrow \fg\mod_{-\chi}^K,
\end{equation} 
where the first arrow is the equivalence \eqref{e:duality g H chi}, and $\ell_K=\Lambda^{\dim(K)}(\fk)[\dim(K)]$. 

\medskip

Let $\BD_{\fg,\chi}^{\on{contr}}$ denote the contravariant equivalence
\begin{equation} \label{e:contr gK}
(\fg\mod_\chi^K)^c\simeq (\fg\mod_{-\chi}^K)^c,
\end{equation} 
corresponding to \eqref{e:to be contr gK}, and similarly for $-\chi$. 

\medskip

By definition, we have
\begin{equation} \label{e:Serre Drinf}
\BD_{\fg,-\chi}^{\on{contr}}\circ \BD_{\fg,\chi}^{\on{can}}\simeq (-\otimes \ell_K) \circ \Se_{\fg\mod_\chi^K}
\end{equation} 
and by \thmref{t:gH miraculous}, 
\begin{equation} \label{e:Drinf Serre}
\BD_{\fg,-\chi}^{\on{can}}\circ \BD_{\fg,\chi}^{\on{contr}}\simeq (-\otimes \ell^{-1}_K) \circ \on{Ps-Id}_{\fg\mod_\chi^K}.
\end{equation} 

\sssec{}

We will now show that $\BD_{\fg,\chi}^{\on{contr}}$ is (the derived version of) the ``usual" contragredient duality 
for Harish-Chandra modules.

\medskip

In particular, we obtain that \eqref{e:Drinf Serre} gives an expression to the
composition of the contragredient duality and ``cohomological" duality on Harish-Chandra modules. In
the context of $\fp$-adic groups, in \cite{BBK}, such a composition is shown to be isomorphic to the 
\emph{Deligne-Lusztig} functor.

\medskip

Hence, we obtain that in the context of Harish-Chandra modules, the pseudo-identity functor $\on{Ps-Id}_{\fg\mod_\chi^K}$
plays a role analogous to that of the Deligne-Lusztig functor for $\fp$-adic groups. 

\medskip

Note that under the localization equivalence (say, when $\chi$ is regular), the functor $\on{Ps-Id}_{\fg\mod_\chi^K}$
corresponds to the pseudo-identity functor $\on{Ps-Id}_{K\backslash X}$. As was mentioned in the introduction, 
certain parallel features of functors of the form $\on{Ps-Id}_{\CY}$ and Deligne-Lusztig type functors have been observed 
elsewhere in geometric representation theory, see, e.g., \cite{DW,Ga2,Wa}.  This analogy is further reinforced by the properties
of the functor $\on{Ps-Id}_{K\backslash X}$ expressed in Conjectures \ref{c:J} and \ref{c:CT} below. 

\sssec{}

Whereas the isomorphism \eqref{e:Drinf Serre} may look somewhat surprising (given the geometric nature of
the functor $\on{Ps-Id}_{\fg\mod_\chi^K}$), the isomorphism \eqref{e:Serre Drinf} is something one could 
have expected, based on the following example (we are grateful to J.~Lurie for pointing this out to us): 

\medskip

Let $A$ be an associative algebra that is finite-dimensional over $k$. On the one hand, dualization over $k$ defines a contravariant
functor
$$\BD^{\on{contr}}_A:A\mod^c\to A^{\on{op}}\mod.$$

\medskip

On the other hand, we have a canonical equivalence
$$(A\mod)^\vee\simeq A^{\on{op}}\mod,$$ 
given by the pairing
$$M',M\mapsto M'\underset{A}\otimes M, \quad A^{\on{op}}\mod\otimes A\mod\to \Vect.$$

Denote the resulting contravariant equivalence by
$$\BD_A^{\on{can}}:A\mod^c\simeq A^{\on{op}}\mod^c.$$

Composing, we obtain a covariant functor
$$\BD_{A^{\on{op}}}^{\on{contr}}\circ \BD_A^{\on{can}}:A\mod^c\to A\mod.$$

It is easy to see that this functor is the restriction to $A\mod^c$ of the Serre functor $\Se_{A\mod}$.

\sssec{}  \label{sss:dual gK}

On the level of abelian categories, we have the "usual" contragredient duality functor
\begin{equation}\label{e:contr dualization gK}
\CM\mapsto \CM^\vee,\quad (\fg\mod_\chi^K)^{\heartsuit}\to ((\fg\mod_{-\chi}^K)^{\heartsuit})^{\on{op}}
\end{equation}
defined by sending $\CM$ to the subspace of the abstract dual $\CM^*$, consisting of $K$-finite vectors. It is an exact functor.

\medskip

Writing $$\CM\simeq \underset{\alpha}\oplus\, \CM_{\alpha},$$ a direct sum of isotypic subspaces with respect to $K$, the module $\CM^{\vee}$ can be described as $$\CM^\vee\simeq \underset{\alpha}\oplus\, (\CM_{\alpha})^*.$$

\sssec{}

It is not hard to show using Localization theory that a module $\CM \in (\fg\mod_{\chi}^K)^{\heartsuit}$ is finitely generated if and only if it is 
of finite length, and that is if and only if it is admissible, where the latter means that each $\CM_{\alpha}$ (notation as above) is finite-dimensional.

\medskip

This shows that $\CM \mapsto \CM^{\vee}$ restricts to a contravariant equivalence 
$$(\fg\mod_\chi^K)^{\heartsuit,\on{f.g.}} \simeq (\fg\mod_{-\chi}^{K})^{\heartsuit,\on{f.g.}}.$$

\sssec{}

We are going to prove: 

\begin{thm}  \label{t:descr contr gK}
The ind-extension of the contravariant equivalence $\BD_{\fg,\chi}^{\on{contr}}$ of \eqref{e:contr gK}
\begin{equation} \label{e:contr ext gK}
\BD_{\fg,\chi}^{\on{contr}}:\fg\mod_\chi^K\to (\fg\mod_{-\chi}^K)^{\on{op}}
\end{equation} 
is t-exact when restricted to $ (\fg\mod_{\chi}^K)^{\on{f.g.}}$. The corresponding contravariant functor $$(\fg\mod_\chi^K)^{\heartsuit, \on{f.g.}} \to ((\fg\mod_{-\chi}^K)^{\heartsuit})^{\on{op}}$$
is given by the functor $\CM\mapsto \CM^\vee$ of \eqref{e:contr dualization gK}.
\end{thm}

\sssec{Proof of \thmref{t:descr contr gK}, Step 1}

We will in fact show that $$\BD_{\fg,\chi}^{\on{contr}} : \fg\mod^K_{\chi} \to (\fg\mod^K_{-\chi})^{op}$$ is t-exact on the whole category, rather than just on $ (\fg\mod_{\chi}^K)^{\on{f.g.}}$.

\medskip

For $\rho\in \on{Irrep}(K)$, set $$P_{\rho , -\chi} := U(\fg)_{-\chi} \underset{U(\fk)}\otimes \rho \in \fg\mod_{-\chi}^K.$$

\medskip

It is clear that
$$\CHom_{\fg\mod_{-\chi}^K}(P_{\rho,-\chi},\CN)\simeq \CHom_{\Rep(K)}(\rho,\CN).$$

Hence, an object $\CN\in \fg\mod_{-\chi}^K$ belongs to $(\fg\mod_{-\chi}^K)^{\leq 0 , \heartsuit , \geq 0}$
if and only if the objects
$$\CHom_{\fg\mod_{-\chi}^K}(P_{\rho,-\chi},\CN)\in \Vect$$
belong to $\Vect^{\leq 0,\heartsuit,\geq 0}$ for all $\rho$.

\medskip

Notice that, by definition, for $\CM \in \fg\mod_{\chi}^K$ and $\CM'\in \fg\mod_{-\chi}^K$ we have
$$\CHom_{\fg\mod_{-\chi}^K}(\CM',\BD_{\fg,\chi}^{\on{contr}}(\CM))\simeq
\langle \CM,\CM'\rangle_{\fg,\chi,K}^*\otimes \ell_K.$$

It thus remains to show that

\begin{equation}   \label{e:functor}
\CM \mapsto \langle \CM,P_{\rho,-\chi} \rangle_{\fg,\chi,K}\otimes \ell_K^{-1}
\end{equation}
is a t-exact functor $ \fg\mod_{\chi}^K \to \Vect$ (for a given $\rho$).

\medskip

We have $$\langle \CM,P_{\rho,-\chi} \rangle_{\fg,\chi,K} \simeq 
\langle \CM,\rho\rangle_{\fk,K}.$$

\medskip

However, according to \propref{p:fin dim dual}
$$\langle \CM,\rho\rangle_{\fk,K}\simeq  \CHom_{\Rep(K)}(\rho^*,\CM)\otimes \ell_K,$$ so that the functor \eqref{e:functor} is isomorphic to $$ \CM \mapsto \CHom_{\Rep(K)}(\rho^*,\CM),$$ which is $t$-exact.

\sssec{Proof of \thmref{t:descr contr gK}, Step 2}

From Step 1, we deduce that the pairing $$ \CM , \CM' \mapsto \langle \CM , \CM' \rangle_{\fg , \chi , K} \otimes \ell_K^{-1}$$ is right t-exact. 
Hence, the comparison map (see \secref{sss:from top}) 
$$ \langle \oblv_K(\CM),\oblv_K(\CM')\rangle_{\fg,\chi} 
\simeq \langle \Av_*^K \circ \oblv_K(\CM) , \CM' \rangle_{\fg , \chi , K} \overset{\text{\eqref{e:top dr}}}\longrightarrow
\langle \CM,\CM'\rangle_{\fg,\chi,K}\otimes \ell_K^{-1} $$ 
induces an isomorphism on $H^0$. 

\medskip

This implies that for $\CM  \in (\fg\mod^K_{\chi} )^{\heartsuit}$ and $\CM' \in (\fg\mod^K_{-\chi} )^{\heartsuit}$, we have:
$$\on{Hom}_{(\fg\mod^K_{-\chi})^{\heartsuit}} (\CM' , \BD_{\fg , \chi}^{\on{contr}} (\CM)) \simeq H^0 \left(  \langle \CM , \CM' \rangle_{\fg , \chi , K} \otimes \ell_K^{-1} \right)^* \simeq $$
$$ \simeq H^0 \left( \langle \oblv_K (\CM) , \oblv_K (\CM') \rangle_{\fg , \chi} \right) ^* = H^0 \left( \CM' \underset{U(\fg)_{\chi}}\otimes \CM \right)^*.$$ 

\medskip

This isomorphism identifies $H^0(\BD^{\on{contr}}_{\fg,\chi} (\CM))$ with $\CM^{\vee}$.

\ssec{The parabolic case}\label{ss:dualmixed}

\sssec{}

We will now consider a hybrid of the situations considered in Sects. \ref{ss:O} and \ref{ss:gK}.
Namely, let $G$ be a reductive group, equipped with an involution $\theta$.
Let $P$ a minimal $\theta$-compatible parabolic, so that $P^-=\theta(P)$ is an opposite parabolic.
Set 
$$M_K=K\cap P\cap P^-.$$

Let $N$ (resp., $N^-$) denote the unipotent radical of $P$ (resp., $P^-$). 

\medskip

We will consider the categories
$$\fg\mod_\chi^{M_K\cdot N} \text{ and } \fg\mod_\chi^{M_K\cdot N^-}.$$

\begin{rem}
The results of this subsection are equally applicable when instead of $P,P^-$ we take any two opposite parabolics and instead of $M_K$ we take the entire Levi subgroup $M=P\cap P^-$.
\end{rem}

\sssec{}

We consider the equivalence 
\begin{equation}   \label{e:to be contr par}
(\fg\mod_\chi^{M_K\cdot N})^\vee \simeq \fg\mod_{-\chi}^{M_K\cdot N^-}
\end{equation} 
given by the composition
$$(\fg\mod_\chi^{M_K\cdot N})^\vee \simeq \fg\mod_{-\chi}^{M_K\cdot N} \overset{\Se_{\fg\mod_{-\chi}^{M_K\cdot N}}}\longrightarrow 
\fg\mod_{-\chi}^{M_K\cdot N}  \overset{\Upsilon^{-1}}\longrightarrow \fg\mod_{-\chi}^{M_K\cdot N^-}
\overset{-\otimes \ell_{M_K}}\simeq \fg\mod_{-\chi}^{M_K\cdot N^-},$$
where the first arrow is the equivalence \eqref{e:duality g H chi}.

\medskip

Let $\BD_{\fg,\chi}^{\on{contr}}$ denote the resulting contravariant equivalence
\begin{equation}   \label{e:contr par}
(\fg\mod_\chi^{M_K\cdot N})^c\to (\fg\mod_\chi^{M_K\cdot N^-})^c
\end{equation}

\medskip

As in \corref{c:prel AG}, using \cite[Proposition 4.1.7]{Kim}, we obtain:  

\begin{prop}  \label{p:parab AG}
The identification \eqref{e:to be contr} is canonically isomorphic to
$$(\fg\mod_\chi^{M_K\cdot N})^\vee \simeq \fg\mod_{-\chi}^{M_K\cdot N} \overset{\Upsilon^-[2\dim(X)-\dim(M_K)]}\longrightarrow 
\fg\mod_{-\chi}^{M_K\cdot N^-}
\overset{-\otimes \ell_{M_K}}\simeq \fg\mod_{-\chi}^{M_K\cdot N^-},$$
where the first arrow is the identification of \eqref{e:duality g H chi}.
\end{prop}

\sssec{}

We define a contravariant functor
\begin{equation}\label{e:contrmixed}
(\fg\mod_\chi^{M_K\cdot N})^{\heartsuit}\to ((\fg\mod_\chi^{M_K\cdot N^-})^{\heartsuit})^{\on{op}}, \quad \CM\mapsto \CM^\vee
\end{equation}
by sending $\CM$ to the subspace of the abstract dual $\CM^*$, consisting of $M_K \cdot N^-$-finite vectors.

\sssec{}

For $\CM \in (\fg\mod_\chi^{M_K\cdot N})^{\heartsuit, \on{f.g.}}$, one has the following concrete description of $\CM^{\vee}$. 

\medskip

Let $\fm$ denote the Lie algebra of the Levi subgroup $M=P\cap P^-$. 
Let $\fa\subset \fm$ denote the $\theta$-split part of the center of $\fm$, so that $\fm\simeq \fm_K\oplus \fa$. 
For $\CM\in (\fg\mod_\chi^{M_K\cdot N})^{\heartsuit,\on{f.g.}}$, the action
of $\fa$ on $\CM$ is locally finite. Hence, we can write
$$\CM \simeq \underset{\mu}\oplus\, \CM_{(\mu)},$$
where the $\CM_{(\mu)}$ are the generalized eigenspaces for the action of $\fa$. 

\medskip

Now, each $\CM_{(\mu)}$ is \emph{admissible} as a $(\fm,K_M)$-module, and let
$(\CM_{(\mu)})^\vee$ denote its dual (taken in the sense of \secref{sss:dual gK}). 
We then have
$$\CM^\vee:=\underset{\mu}\oplus\, (\CM_{(\mu)})^\vee,$$
with the natural action of $\fg$. 

\medskip

Again it is possible to show that the functor $\CM \mapsto \CM^{\vee}$ restricts to a contravariant equivalence 
$$(\fg\mod_\chi^{M_K\cdot N})^{\heartsuit, \on{f.g.}} \simeq (\fg\mod_\chi^{M_K\cdot N^-})^{\heartsuit, \on{f.g.}}.$$

\sssec{}

We claim:

\begin{thm}  \label{t:descr contr par}
The ind-extension of the contravariant equivalence $\BD_{\fg,\chi}^{\on{contr}}$ of \eqref{e:contr par}
$$\BD_{\fg,\chi}^{\on{contr}}:\fg\mod_\chi^{M_K\cdot N}\to (\fg\mod_{-\chi}^{M_K\cdot N^-})^{\on{op}}$$
is t-exact when restricted to $(\fg\mod_\chi^{M_K\cdot N})^{\on{f.g.}}$.
The corresponding contravariant functor 
$$(\fg\mod_\chi^{M_K\cdot N})^{\heartsuit, \on{f.g.}} \to ((\fg\mod_{-\chi}^{M_K\cdot N^-})^\heartsuit)^{\on{op}}$$
is given by the functor $\CM\mapsto \CM^\vee$ of \eqref{e:contrmixed}. 
\end{thm}

We omit the proof as it is obtained by combining the ideas in the proofs of 
Theorems \ref{t:descr contr} and \ref{t:descr contr gK}. 

\section{Relation to the ``2nd adjointness" conjecture}  \label{s:2nd adj}

In this section we recall the ``2nd adjointness" conjecture of \cite{Kim} and relate it to \thmref{t:descr contr gK}. The notation is as in \secref{ss:dualmixed}.

\ssec{The principal series functors}

\sssec{}

Consider the categories $\fg\mod_\chi^{M_K\cdot N}$, $\fg\mod_\chi^{M_K\cdot N^-}$, $\fg\mod_\chi^K$.
As was shown in \cite[Theorem 4.2.2]{Kim}, the partially defined functors
$$\Av^N_!:\fg\mod^{M_K}_\chi\to \fg\mod_\chi^{M_K\cdot N} \text{ and } 
\Av^{N^-}_!:\fg\mod^{M_K}_\chi\to \fg\mod_\chi^{M_K\cdot N^-}$$
are defined on the essential image of $\oblv_{K/M_K}$. 

\medskip

In particular, we have an adjoint pair
$$\Av^N_!\circ \oblv_{K/M_K}:\fg\mod_\chi^K\rightleftarrows \fg\mod_\chi^{M_K\cdot N}:\Av^{K/M_K}_*\circ \oblv_N.$$

\medskip

In addition, we have:

\begin{prop} \hfill

\smallskip

\noindent{\em(a)}
For $\bC$ being either $\fg\mod_\chi$ or $\Dmod_\lambda(X)$, 
the partially defined functor 
$$\Av^{K/M_K}_!:\bC^{M_K}\to \bC^K$$
is defined on the essential image of $\oblv_N:\bC^{M_K\cdot N}\to \bC^{M_K}$. 

\smallskip

\noindent{\em(b)}
Moreover, we have
$$\Av^{K/M_K}_!\simeq \Gamma \circ \Av^{K/M_K}_!\circ \Loc.$$

\end{prop}

\begin{proof}
The functor $\Av^{K/M_K}_!:\Dmod_\lambda(X)^{M_K}\to \Dmod_\lambda(X)^K$ is defined
on the essential image of $\oblv_N:\Dmod_\lambda(X)^{M_K\cdot N}\to  \Dmod_\lambda(X)^{M_K}$
because $M_K\cdot N$ has finitely many orbits on $X$, and hence every object from 
$\Dmod_\lambda(X)^{M_K\cdot N}$ is holonomic. 

\medskip

The assertion concerning $\fg\mod_\chi$, as well as point (b) of the proposition follow from
\cite[Proposition 1.2.6]{Kim}.

\end{proof} 

In particular, we obtain another pair of adjoint functors
$$\Av^{K/M_K}_!\circ \oblv_N:\fg\mod_\chi^{M_K\cdot N}\rightleftarrows \fg\mod_\chi^K:\Av^N_*\circ \oblv_{K/M_K}.$$

\begin{rem}
We regard the functors
$$\Av^{K/M_K}_!\circ \oblv_N \text{ and } \Av^{K/M_K}_*\circ \oblv_N,$$
which map $\fg\mod_\chi^{M_K\cdot N}\to \fg\mod_\chi^K$, 
as two versions of the principal series functor for Harish-Chandra modules. 
\end{rem} 

\sssec{}

We now claim:

\begin{prop}   \label{p:Av ! and *}
There exists a canonical isomorphism of functors
$$\Av^{K/M_K}_!\simeq \PsId_{\fg\mod_\chi^K}\circ \Av^{K/M_K}_*[2\dim(X)-\dim(M_K)], \quad 
\fg\mod_\chi^{M_K\cdot N}\rightrightarrows \fg\mod_\chi^K.$$
\end{prop}

\begin{proof}

The point of departure is the isomorphism 
$$\Av^{K/M_K}_!\simeq \PsId_{K\backslash X}\circ \Av^{K/M_K}_*[2\dim(X)-\dim(M_K)]$$
as functors 
$$\Dmod_\lambda(X)^{M_K\cdot N}\rightrightarrows \Dmod_\lambda(X)^K$$
that was established in \cite[Corollary 4.4.2]{Kim}.

\medskip

Using \thmref{t:Y miraculous}, we rewrite it as an isomorphism 
$$\Se_{\Dmod_\lambda(X)^K}\circ \Av^{K/M_K}_!\simeq \Av^{K/M_K}_*[2\dim(X)-\dim(M_K)].$$

\medskip

Composing both sides with $\Gamma$ and precomposing with $\Loc$, we obtain an isomorphism between
the functor
$$\Av^{K/M_K}_*[2\dim(X)-\dim(M_K)]: \Dmod_\lambda(X)^{M_K\cdot N}\rightrightarrows \Dmod_\lambda(X)^K$$
and 

$$\Gamma \circ \Se_{\Dmod_\lambda(X)^K} \circ \Av^{K/M_K}_!\circ \Loc\simeq 
\Gamma \circ \Se_{\Dmod_\lambda(X)^K} \circ \Loc  \circ \Av^{K/M_K}_! \simeq 
\Se_{\fg\mod_\chi^K} \circ \Av^{K/M_K}_!,$$
where the first isomorphism is due to  \cite[Lemma 1.2.6]{Kim}, and the second one is using \lemref{l:double right}(b).

\medskip

Thus, we obtain an isomorphism
$$\Se_{\fg\mod_\chi^K}\circ \Av^{K/M_K}_!\simeq \Av^{K/M_K}_*[2\dim(X)-\dim(M_K)].$$

Applying \thmref{t:gH miraculous}(a), we arrive at the desired isomorphism 
$$\Av^{K/M_K}_!\simeq \PsId_{\fg\mod_\chi^K}\circ \Av^{K/M_K}_*[2\dim(X)-\dim(M_K)].$$

\end{proof}

\ssec{The ``2nd adjointness" conjecture}

\sssec{}

Let us recall the ``2nd adjointness" conjecture of \cite[Conjecture 4.4.5]{Kim}:

\begin{conj}  \label{c:2nd adj}
The functors
$$\Av^{K/M_K}_* \text{ and } (-\otimes \ell_{K/M_K}^{-1})\circ \Av^{K/M_K}_!\circ \Upsilon,\quad 
\fg\mod_\chi^{M_K\cdot N^-}\rightrightarrows \fg\mod_\chi^K$$
are canonically isomorphic, where $\ell_{K/M_K}:=\ell_K\otimes \ell^{-1}_{M_K}$.
\end{conj}

\begin{rem}
One should think of the line $\ell_{K/M_K}^{-1}$ as the top de Rham cohomology of $K/M_K$.
\end{rem}

\sssec{}

Using \propref{p:Av ! and *}, we can reformulate \conjref{c:2nd adj} as follows:

\begin{conj} \label{c:funct eq}
The following diagram of functors commutes:
$$
\CD
\fg\mod_\chi^K  @<{\Av^{K/M_K}_*}<< \fg\mod_\chi^{M_K\cdot N} \\
@V{(-\otimes \ell_{K/M_K}^{-1}) \circ \on{Ps-Id}_{\fg\mod_\chi^K}}VV   @VV{ \Upsilon^{-1}[-2\dim(X)+\dim(M_K)]}V  \\
\fg\mod_\chi^K  @<{\Av^{K/M_K}_*}<< \fg\mod_\chi^{M_K\cdot N^-}.
\endCD
$$
\end{conj} 

\sssec{}

Yet another reformulation of  \conjref{c:2nd adj} is the following:

\begin{conj}  \label{c:J}
The functor \emph{right} adjoint to   
$$\Av^{K/M_K}_*\circ \oblv_{N^-}: \fg\mod_\chi^{M_K\cdot N^-}\to \fg\mod_\chi^K,$$
is given by
$$(-\otimes \ell_{K/M_K})\circ J\circ \oblv_{K/M_K},$$
\end{conj}

In \conjref{c:J}, the notation $J$ stands for the Casselman-Jacquet functor (see \cite{Kim}, where this functor
is studied in detail).  In what follows we will write $J$ instead of $J\circ \oblv_{K/M_K}$, and mean by it the
corresponding functor
$$\fg\mod_\chi^K\to \fg\mod_\chi^{M_K\cdot N^-}.$$

We recall that according to \cite[Theorem 4.2.2]{Kim}, we have
\begin{equation} \label{e:J}
J\simeq \Av^{N^-}_!\circ \oblv_N\circ \Av^N_* \circ \oblv_{K/M_K} \simeq 
\Av^{N^-}_*\circ \oblv_N\circ \Av^N_! \circ \oblv_{K/M_K},
\end{equation} 
where the last isomorphism expressed the ``Verdier self-duality" property of the functor $J$. 

\medskip

Here is yet another equivalent formulation of \conjref{c:2nd adj}:

\begin{conj}  \label{c:CT}
The functors 
$$(-\otimes \ell_{K/M_K})\circ J[-2\dim(X)+\dim(M_K)] \text{ and } \Av^{N^-}_*\circ \oblv_{K/M_K}\circ \on{Ps-Id}_{\fg\mod_\chi^K}$$
from $\fg\mod_\chi^K$ to $\fg\mod_\chi^{M_K\cdot N^-}$
are canonically isomorphic.
\end{conj}

\begin{proof}
We start with the isomorphism 
$$\on{Ps-Id}_{\fg\mod_\chi^K}\circ \Av^{K/M_K}_*\circ \Upsilon^- \simeq  (-\otimes \ell_{K/M_K})\circ \Av^{K/M_K}_*[-2\dim(X)+\dim(M_K)]$$
that follows from \conjref{c:funct eq}. Passing to dual functors with respect to \eqref{e:duality g H chi}, we obtain
$$\Upsilon\circ \Av^{N^-}_* \circ \oblv_{K/M_K}\circ \on{Ps-Id}_{\fg\mod_\chi^K}\simeq  (-\otimes \ell_{K/M_K})\circ
\Av^N_* \circ \oblv_{K/M_K}[-2\dim(X)+\dim(M_K)].$$

Now we use the fact that
$$J\simeq \Upsilon^{-1}\circ  \Av^N_*.$$
\end{proof}

\begin{rem}
Note that Conjectures \ref{c:J} and \ref{c:CT} further reinforce the analogy between the functor $\on{Ps-Id}_{\fg\mod_\chi^K}$
and the Deligne-Lusztig functor for $\fp$-adic groups.
\end{rem} 

\ssec{A plausibility check}

We will now juxtapose \conjref{c:J} with Theorems \ref{t:descr contr gK} and \ref{t:descr contr par}, 
and arrive at a certain plausible (and at the level of abelian categories, known) statement.

\sssec{}

Recall (see \eqref{e:J}) that the functor
$$J:\fg\mod_\chi^K\to \fg\mod_\chi^{M_K\cdot N^-}$$
is isomorphic to 
$$\fg\mod_\chi^K \overset{\Av^N_!\circ \oblv_{K/M_K}}\longrightarrow \fg\mod_\chi^{M_K\cdot N}
\overset{\Upsilon^-}\longrightarrow  \fg\mod_\chi^{M_K\cdot N^-}.$$

In particular, it preserves compactness. Let 
$$J^{\on{op}}:(\fg\mod_\chi^K)^\vee \to (\fg\mod_\chi^{M_K\cdot N^-})^\vee$$
denote the corresponding \emph{opposite} functor. I.e., this is the ind-extension of the functor
$$((\fg\mod_\chi^K)^\vee)^c \simeq ((\fg\mod_\chi^K)^c)^{\on{op}}\overset{J^{\on{op}}}\longrightarrow 
((\fg\mod_\chi^{M_K\cdot N^-})^c)^{\on{op}}\simeq ((\fg\mod_\chi^{M_K\cdot N^-})^\vee)^c.$$

\medskip

The isomorphism in \eqref{e:J} implies that with respect to the canonical identifications
$$(\fg\mod_\chi^K)^\vee  \simeq \fg\mod_{-\chi}^K \text{ and }
(\fg\mod_\chi^{M_K\cdot N^-})^\vee  \simeq \fg\mod_{-\chi}^{M_K\cdot N^-}$$
of \eqref{e:duality g H chi}, we have a commutative diagram
\begin{equation}  \label{e:com diag 1}
\CD
(\fg\mod_\chi^K)^\vee  @>{\text{\eqref{e:duality g H chi}}}>>  \fg\mod_{-\chi}^K  \\
@V{J^{\on{op}}}VV    @VV{J}V  \\
(\fg\mod_\chi^{M_K\cdot N^-})^\vee  @>{\text{\eqref{e:duality g H chi}}}>>    \fg\mod_{-\chi}^{M_K\cdot N^-}. 
\endCD
\end{equation}

In other words, we have an isomorphism as contravariant functors  
$$\BD^{\on{can}}_{\fg,\chi}\circ J \simeq J\circ \BD^{\on{can}}_{\fg,\chi}, \quad 
(\fg\mod_\chi^K)^c\rightrightarrows (\fg\mod_{-\chi}^{M_K\cdot N^-})^c$$

\sssec{}

Let
$$J^-:\fg\mod_\chi^K\to \fg\mod_\chi^{M_K\cdot N}$$
be the variant of the functor $J$, where we swap the roles of $N$ and $N^-$. 

\medskip

We will now show that the following conjecture is equivalent to \conjref{c:2nd adj}:

\begin{conj} \label{c:Jacquet duality} 
The following diagram of functors commutes
\begin{equation}  \label{e:com diag 2}
\CD
(\fg\mod_\chi^K)^\vee   @>{\text{\eqref{e:to be contr gK}}}>>  \fg\mod_{-\chi}^K  \\
@V{J^{\on{op}}}VV  @VV{J^-}V  \\
(\fg\mod_\chi^{M_K\cdot N^-})^\vee  @>{\text{\eqref{e:to be contr par}}}>>    \fg\mod_{-\chi}^{M_K\cdot N}.
\endCD
\end{equation}
Equivalently, we have an isomorphism of the contravariant functors  
\begin{equation} \label{e:J and contr}
\BD^{\on{contr}}_{\fg,\chi}\circ J \simeq J^-\circ \BD^{\on{contr}}_{\fg,\chi}, 
\quad 
(\fg\mod_\chi^K)^c\rightrightarrows (\fg\mod_{-\chi}^{M_K\cdot N})^c.
\end{equation}
\end{conj}

\begin{rem}
Note that the analog of \eqref{e:J and contr} for admissible representations of $\fp$-adic group
is a known statement, which can also be easily deduced from the 2nd adjointness theorem.
\end{rem} 

\begin{proof}

Juxtaposing the diagrams \eqref{e:com diag 1} and \eqref{e:com diag 2} and using \cite[Theorem 4.1.7]{Kim},
we obtain that \conjref{c:Jacquet duality}
is equivalent to the commutation of the diagram
$$
\CD
\fg\mod_\chi^K  @>{(-\otimes \ell_K)\circ \Se_{\fg\mod_\chi^K}}>>  \fg\mod_\chi^K \\
@V{J}VV   @VV{J^-}V   \\ 
\fg\mod_\chi^{M_K\cdot N^-}  @>{(-\otimes \ell_{M_K})\circ \Upsilon[2\dim(X)-\dim(M_K)]}>> \fg\mod_\chi^{M_K\cdot N}.  
\endCD
$$

Using the fact that $J\simeq \Upsilon^{-1}\circ \Av^N_*\circ \oblv_{K/M_K}$, the commutativity of the above
diagram is equivalent to the isomorphism
$$\Av^N_*\circ \oblv_{K/M_K}[2\dim(X)-\dim(M_K)]\simeq (-\otimes \ell_{K/M_K})\circ J^-\circ \Se_{\fg\mod_\chi^K},$$
or equivalently
$$\Av^N_*\circ \oblv_{K/M_K}\circ \on{Ps-Id}_{\fg\mod_\chi^K}\simeq (-\otimes \ell_{K/M_K})\circ J^-[-2\dim(X)+\dim(M_K)],$$
while the latter is \conjref{c:CT}.

\end{proof}

\sssec{}

Consider the ind-extensions of the (contravariant) functors 
$\BD^{\on{contr}}_{\fg,\chi}\circ J$ and  $J^-\circ \BD^{\on{contr}}_{\fg,\chi}$, appearing in \conjref{c:Jacquet duality}.
We obtain two functors
$$\fg\mod_\chi^K\rightrightarrows (\fg\mod_{-\chi}^{M_K\cdot N})^{\on{op}},$$
and let us restrict both to 
$$(\fg\mod_\chi^K)^{\on{f.g.}}\subset \fg\mod_\chi^K.$$

\medskip

We note:

\begin{lem}
The functor $J$ (resp., $J^-$) maps $(\fg\mod_\chi^K)^{\on{f.g.}}$ to $(\fg\mod_\chi^{M_K\cdot N^-})^{\on{f.g.}}$ 
(resp., $(\fg\mod_\chi^{M_K\cdot N})^{\on{f.g.}}$).
\end{lem} 

\begin{proof}

According to \cite[Theorem 4.4.2(a)]{Kim}, the functor $J$ is t-exact. Hence, it is enough to show that it sends 
$(\fg\mod_\chi^K)^{\heartsuit,\on{f.g.}}$ to $(\fg\mod_\chi^{M_K\cdot N^-})^{\heartsuit,\on{f.g.}}$.

For that, it is sufficient to show that for every object $\CM\in (\fg\mod_\chi^K)^{\heartsuit,\on{f.g.}}$, there
exists an object $\CM'\in ((\fg\mod_\chi^K)^{\on{f.g.}})^{\leq 0}$ with $H^0(\CM')\simeq \CM$, such
that $J(\CM')\in (\fg\mod_\chi^{M_K\cdot N^-})^{\on{f.g.}}$.

\medskip

However, for every $\CM$ as above, we can choose $\CM'$ compact, such that $H^0(\CM')\simeq \CM$. 
The claim now follows since $J$ preserves compactness and

$$(\fg\mod_\chi^{M_K\cdot N^-})^c\subset (\fg\mod_\chi^{M_K\cdot N^-})^{\on{f.g.}}.$$

\end{proof}

\sssec{}

Thus, we obtain two contravariant functors
$$(\fg\mod_\chi^K)^{\on{f.g.}}\to (\fg\mod_{-\chi}^{M_K\cdot N})^{\on{f.g.}},$$
which according to Theorems \ref{t:descr contr gK} and \ref{t:descr contr par}, 
and \cite[Theorem 4.4.2(a)]{Kim} are t-exact.

\medskip

Moreover, the corresponding functors
$$(\fg\mod_\chi^K)^{\heartsuit,\on{f.g.}}\rightrightarrows (\fg\mod_{-\chi}^{M_K\cdot N})^{\heartsuit,\on{f.g.}},$$
identify with
\begin{equation} \label{e:Cass}
\CM\mapsto J(\CM)^\vee \text{ and } \CM\mapsto J^-(\CM^\vee),
\end{equation}
respectively. 

\medskip

Now, the isomorphism between the functors \eqref{e:Cass} is known, when the ground field $k$ equals 
$\BC$. Namely, W. Casselman constructed the map $J^- (\CM^{\vee}) \to J(\CM)^{\vee}$ (see. for example, \cite{Ca}), and D. Milicic, and later H. Hecht and W. Schmid, showed that it is an isomorphism (see \cite{M}, \cite{HS}). The constructions and methods are analytic, using the asymptotic expansion of matrix coefficients.

\end{document}